\documentclass[a4paper]{amsart}

\usepackage[all,cmtip]{xy}
\usepackage{a4wide}
\usepackage{verbatim}
\usepackage{amsmath}
\usepackage{amssymb}
\usepackage{amsthm}
\usepackage{latexsym}
\usepackage{enumerate}
\usepackage{prettyref}
\usepackage{hyperref}
\usepackage{bm}
\usepackage{mathrsfs}
\usepackage[usenames,dvipsnames]{color}
\definecolor{dark-blue}{rgb}{0,0,0.5}
\usepackage{tikz}
\usetikzlibrary{matrix,arrows}
\usepackage[abbrev,lite,alphabetic]{amsrefs}
\usepackage{dsfont}
\usetikzlibrary{matrix,arrows}
\usepackage{tikz-cd}

\theoremstyle{plain}
\newtheorem{theorem}[subsection]{Theorem}
\newrefformat{thm}{\hyperref[{#1}]{Theorem~\ref*{#1}}}
\newtheorem{lemma}[subsection]{Lemma}
\newrefformat{lem}{\hyperref[{#1}]{Lemma~\ref*{#1}}}
\newtheorem{proposition}[subsection]{Proposition}
\newrefformat{prop}{\hyperref[{#1}]{Proposition~\ref*{#1}}}
\newtheorem{corollary}[subsection]{Corollary}
\newrefformat{cor}{\hyperref[{#1}]{Corollary~\ref*{#1}}}
\newtheorem{conjecture}[subsection]{Conjecture}
\newtheorem{notation}[subsection]{Notation}

\theoremstyle{remark}

\newtheorem{remark}[subsection]{Remark}
\newrefformat{rem}{\hyperref[{#1}]{Remark~\ref*{#1}}}
\newtheorem{examplecore}[subsection]{Example}
\newrefformat{ex}{\hyperref[{#1}]{Example~\ref*{#1}}}
\newtheorem{construction}[subsection]{Construction}
\newrefformat{cons}{\hyperref[{#1}]{Construction~\ref*{#1}}}

\theoremstyle{definition}
\newtheorem{definition}[subsection]{Definition}
\newrefformat{def}{\hyperref[{#1}]{Definition~\ref*{#1}}}

\newenvironment{example}{\begin{examplecore}}{\hspace*{\fill}
$\square$\par\vspace{.1cm}\end{examplecore}}

\newcommand{\op}{\operatorname}
\newcommand{\et}{\mathrm{\acute{e}t}}

\newcommand{\Spec}{\operatorname{Spec}}

\begin{document}

\title{$\mathbb{A}^1$-connected components of classifying spaces and purity for torsors}  

\author{Elden Elmanto, Girish Kulkarni and Matthias Wendt}

\date{March 2021}

\address{Elden Elmanto, Department of Mathematics, Harvard University, 1 Oxford St. Cambridge, MA 02138, USA}
\email{elmanto@math.harvard.edu}

\address{Girish Kulkarni, Fachgruppe Mathematik und Informatik, Bergische Universit\"at Wuppertal, Gau\ss{}\-strasse 20, 42119 Wuppertal, Germany}
\email{kulkarni@uni-wuppertal.de}

\address{Matthias Wendt, Fachgruppe Mathematik und Informatik, Bergische Universit\"at Wuppertal, Gau\ss{}\-strasse 20, 42119 Wuppertal, Germany}
\email{m.wendt.c@gmail.com}

\subjclass[2010]{}
\keywords{}

\begin{abstract}
  In this paper, we study the Nisnevich sheafification $\mathcal{H}^1_{\et}(G)$ of the presheaf associating to a smooth scheme the set of isomorphism classes of $G$-torsors, for a reductive group $G$. We show that if $G$-torsors on affine lines are extended, then $\mathcal{H}^1_{\et}(G)$ is homotopy invariant and show that the sheaf is unramified if and only if Nisnevich-local purity holds for $G$-torsors. We also identify the sheaf $\mathcal{H}^1_{\et}(G)$ with the sheaf of $\mathbb{A}^1$-connected components of the classifying space ${\rm B}_{\et}G$. This establishes the homotopy invariance of the sheaves of components as conjectured by Morel. It moreover provides a computation of the sheaf of $\mathbb{A}^1$-connected components in terms of unramified $G$-torsors over function fields whenever Nisnevich-local purity holds for $G$-torsors.
\end{abstract}

\maketitle
\setcounter{tocdepth}{1}
\tableofcontents

\section{Introduction}

\subsection{The Asok--Fasel program} One of the most compelling recent success in applying methods from homotopy theory to algebraic geometry is the Asok--Fasel program which attempts to classify vector bundles on smooth affine varieties by means of the Postnikov tower in Morel--Voevodsky's $\mathbb{A}^1$-homotopy theory \cites{asok-fasel-spheres, asok-fasel-3folds, asok-fasel-splitting}. Their methodology is based on the affine representability results of Morel \cite{MField} and Asok, Hoyois and the third author \cite{affine-representability-1} which establish results to the effect that vector bundles on various smooth affine schemes are computed as maps in the $\mathbb{A}^1$-homotopy category, thus accessible via the Postnikov tower \cites{morel:voevodsky,antieau:elmanto} and therefore understood in cohomological terms. A morphism
\[
\mathcal{X} \rightarrow \mathcal{Y}
\]
in the $\mathbb{A}^1$-homotopy category can be understood as the limit of a sequence of maps
\[
\mathcal{X} \rightarrow \tau_{\leq j}\mathcal{Y} \qquad j \geq 0,
\]
such that the obstruction to the extension problem:

\begin{equation}\label{eq:stage}
\begin{tikzcd} 
 & \tau_{\leq {j+1}}\mathcal{Y} \ar{d} \\
\mathcal{X} \ar[dashed]{ur}\ar{r} & \tau_{\leq j}\mathcal{Y},
\end{tikzcd}
\end{equation}
and their possible solutions lie in a Nisnevich cohomology group.

These ``affine representability" results, in turn, rely on the resolution of the Bass--Quillen conjecture in various cases, due to work of Lindel \cite{lindel} and Popescu \cite{popescu:polynomials}. More specifically, let us write $\mathcal{V}_r(R)$ to be the set of isomorphism classes of projective modules of rank $r$ on a ring $R$ (equivalently, rank $r$ vector bundles on $\Spec R$). It is proved in \cite{affine-representability-1}*{Theorem 5.2.1} that whenever $A$ is a ring such that for any maximal ideal $\mathfrak{m}$, $A_{\mathfrak{m}}$ is ind-smooth over a Dedekind ring with perfect residue fields, the map
\[
\mathcal{V}_r(A) \rightarrow \mathcal{V}_r(A[t_1,\cdots, t_j]),
\]
for any $j, r \geq 0$ is an isomorphism. Thus, in particular, $\mathbb{A}^1$-invariance  holds for vector bundles on smooth affines schemes over perfect fields.

Since $\mathcal{V}_r(A)$ is isomorphic to the set of ${\rm GL}_r$-torsors on $\Spec A$, it is natural to ask if the Asok--Fasel program can be used to study torsors under other reductive groups. For us, $G$-torsors are required to have local triviality with respect to the \'etale topology and hence isomorphism classes of $G$-torsors over a scheme $X$ are in (pointed) bijection with the (pointed) Galois cohomology set ${\rm H}^1_{\et}(X;G)$ (pointed at the isomorphism class of the  constant torsor $G \times X \rightarrow X$). The analogous picture for ${\rm H}^1_{\rm Nis}(X;G)$ (where we require that torsors exhibit \emph{Nisnevich-local} triviality) has been discussed in \cite{affine-representability-2} where affine representability was proved for isotropic reductive groups. In particular, whenever the Nisnevich and \'etale cohomological sets coincide (for example, when $G = {\rm SL}_n, {\rm GL}_n, {\rm Sp}_{2n})$, we can run the Asok--Fasel program to classify torsors under these groups. However, this coincidence between \'etale and Nisnevich torsors is not the ``generic situation", and \'etale-locally trivial torsors are certainly more abundant in algebraic geometry --- for example Nisnevich-locally trivial torsors over fields are all constant, a fact which is wildly false for \'etale-locally trivial torsors, e.g. for groups like ${\rm O}(n)$ (related to quadratic forms) or ${\rm PGL}_n$ (related to central simple algebras).

On first glance, it seems that the Asok--Fasel program will not have much to say about \'etale--locally trivial torsors (although frequently ``stably trivial'' torsors happen to be Nisnevich-locally trivial). Indeed, we have known for a long time that on a fixed affine scheme $X$, isomorphism classes of torsors under an arbitrary reductive group over $X$ are not generally $\mathbb{A}^1$-invariant, i.e., the map of pointed sets
\begin{equation}\label{eq:bass-quillen}
{\rm H}^1_{\et}(X, G) \rightarrow {\rm H}^1_{\et}(X \times \mathbb{A}^1, G),
\end{equation}
need not be an isomorphism. Since pullback along the zero section provides a section of the map~\eqref{eq:bass-quillen} what is at stake is the surjectivity of this map, i.e., whether or not a $G$-torsor on $X \times \mathbb{A}^1$ is \emph{extended} from $X$. In other words, the analog of the Bass--Quillen conjecture fails and there is no hope to represent torsors under $G$ as maps in the $\mathbb{A}^1$-homotopy category.

Our first result states that, in fact, the failure of $\mathbb{A}^1$-invariance \emph{disappears} after Nisnevich sheafification. To formulate it, we fix a field $k$ and let us write $\mathcal{H}^1_{\et}(G)$ for the Nisnevich sheafification of the presheaf on ${\rm Sm_k}$ given by
\[
U \mapsto {\rm H}^1_{\et}(U;G).
\]
\begin{theorem}[Propositions~\ref{prop:hoinv-h1etg},~\ref{prop:hoinv-on},~\ref{prop:hoinv-pgln}]\label{thm:cex}   Let $k$ be a field and $G$ be a $k$-group scheme which is either:
\begin{enumerate}
\item a connected reductive group if $k$ is of characteristic zero, or
\item ${\rm O}(n)$ if $k$ is characteristic not two, or
\item ${\rm PGL}_n$ if the characteristic of $k$ is coprime to $n$.
\end{enumerate}
Then, $\mathcal{H}^1_{\et}(G)$ is $\mathbb{A}^1$-invariant. 
\end{theorem}

To put our result in context, we note the following

\begin{example}[Parimala] \label{exam:parimala} In the remarkable paper \cite{parimala} constructed infinitely many, explicit rank $4$ inner product spaces on $\mathbb{R}[x,y]$ which are not extended from $\mathbb{R}$. Translating this problem to torsors, she showed that the map
\[
{\rm H}^1_{\et}(\mathbb{R} ;{\rm O}(4)) \rightarrow {\rm H}^1_{\et}(\mathbb{R}[x,y]; {\rm O}(4))
\]
is \emph{not} an isomorphism. Furthermore, by the last paragraph of \cite{parimala}*{Page 923} and the exceptional isomorphism ${\rm PGL}_2 = {\rm SO}(3)$, we can further conclude that the map
\[
{\rm H}^1_{\et}(\mathbb{R}; {\rm PGL}_2) \rightarrow {\rm H}^1_{\et}(\mathbb{R}[x,y]; {\rm PGL}_2)
\]
is \emph{not} an isomorphism. The reader is invited to consult \cite{lam:book}*{Chapter VII} for further discussions.

In fact, Parimala proved more. She showed that her torsors are constant, and actually extended from $\mathbb{R}$, away from a principal open $\mathbb{A}^2_{\mathbb{R}}$, see \cite{lam:book}*{Lemma VII.4.14}. Additionally, restricting her torsor to a local ring around each point of $\mathbb{A}^2_{\mathbb{R}}$ gives rise to a torsor which is extended from $\mathbb{R}$ \cite{parimala}*{Theorem 2.1(ii)}. These addendum to her results encouraged us to look for a result of the form Theorem~\ref{thm:cex}.

\end{example} 

%
%
%

Theorem~\ref{thm:cex} implies that, in particular, any counterexample to the surjectivity of~\eqref{eq:bass-quillen} must vanish on a Nisnevich cover of $X$. Theorem~\ref{thm:cex} is proved by the following principle: the validity of the Grothendieck--Serre conjecture for $G$ plus $\mathbb{A}^1$-invariance over fields implies Nisnevich-local $\mathbb{A}^1$-invariance; see Proposition~\ref{prop:lnis-hoinv} for a precise statement. 

\subsection{Towards motivic $G$-torsors and a conjecture of Morel's} To connect with the Asok--Fasel program, Theorem~\ref{thm:cex} serves as an input to the next result which belongs to $\mathbb{A}^1$-homotopy theory. To state it, let us recall some notation. Denote by ${\rm B}_{\et}G$ the classifying stack for the group $G$, so that evaluating this stack on a scheme $X$, the pointed set of connected component of the groupoid ${\rm B}_{\et}G(X)$ is canonically isomorphic to ${\rm H}^1_{\et}(X;G)$ (as pointed sets). Restricting ${\rm B}_{\et}G$ to the category of smooth schemes over a fixed field $k$ we can takes its motivic localization $L_{\rm mot}{\rm B}_{\et}G$ so that if $X$ is a smooth $k$-scheme, then the set of homotopy classes of maps from $X$ into $L_{\rm mot}{\rm B}_{\et}G$ computes maps in the $\mathbb{A}^1$-homotopy category: 
\[
\pi_0(L_{\rm mot}{\rm B}_{\et}G(X)) \cong [X, {\rm B}_{\et}G]_{\mathbb{A}^1}.
\]
We have a canonical map
\[
{\rm H}^1_{\et}(X, G) \rightarrow [X, B_{\et}G]_{\mathbb{A}^1} = \pi_0^{\mathbb{A}^1}(B_{\et}G)(X).
\]
which induces a canonical map of Nisnevich sheaves
\begin{equation}\label{eq:compare-map}
\mathcal{H}^1_{\et}(G)\to a_{\rm Nis}\pi_0^{\mathbb{A}^1}({\rm B}_{\et}G).
\end{equation}
In general, if $\mathcal{X}$ is a presheaf of spaces on ${\rm Sm}_k$, the Nisnevich sheaf $a_{\rm Nis}\pi_0^{\mathbb{A}^1}(\mathcal{X})$ is the \emph{sheaf of $\mathbb{A}^1$-connected components of $\mathcal{X}$}. A conjecture of Morel posits that $a_{\rm Nis}\pi_0^{\mathbb{A}^1}(\mathcal{X})$ is an $\mathbb{A}^1$-invariant sheaf. The next theorem furnishes new examples for the validity of this conjecture.

\begin{theorem}\label{thm:morel-g} Let $k$ be a field and $G$ be a group satisfying the hypotheses of Theorem~\ref{thm:cex}. Then:
\begin{enumerate}
\item the canonical map $\mathcal{H}^1_{\et}(G)\to a_{\rm Nis}\pi_0^{\mathbb{A}^1}({\rm B}_{\et}G)$ is an isomorphism of Nisnevich sheaves of pointed sets;
\item the sheaf $a_{\rm Nis}\pi_0^{\mathbb{A}^1}({\rm B}_{\et}G)$ is $\mathbb{A}^1$-invariant and is unramified in the sense of \cite{MField}*{Definition 2.1}. 
\end{enumerate}
\end{theorem}

We make several remarks concerning Theorem~\ref{thm:morel-g}.

\begin{remark}\label{rem:motivic-torsor} A \emph{motivic $G$-torsor} is a map $X \rightarrow {\rm B}_{\et}G$, i.e., an element of $[X, {\rm B}_{\et}G]_{\mathbb{A}^1}$. Theorem~\ref{thm:morel-g} provides the first step towards the classification of motivic $G$-torsors via Postnikov tower methods and gives an algebro-geometric meaning to such a classification. Indeed, in contrast to the situation for vector bundles where the sheaf of $\mathbb{A}^1$-connected components is trivial (which is basically a consequence of ``Hilbert theorem 90"), we see that ${\rm B}_{\et}G$ is \emph{not} $\mathbb{A}^1$-connected. Hence, the first stage of the lifting problem of~\eqref{eq:stage} requires that we choose a map $X \rightarrow a_{\rm Nis}\pi_0^{\mathbb{A}^1}({\rm B}_{\et}G)$ which, after Theorem~\ref{thm:morel-g} is a Nisnevich-local $G$-torsor. In many cases, we can say more explicitly what this means. For example, by Proposition~\ref{prop:components-on}, when $ G = {\rm O}(n)$ this is the data of an unramified element of the Grothendieck-Witt group of $X$ with a rank $n$-representative. 
\end{remark}

\begin{remark}\label{rem:challenge} Theorem~\ref{thm:morel-g} also introduces the main challenge in attempting to run the Asok--Fasel program to classify motivic $G$-torsors. Since the classifying stack for ${\rm GL}_r$-torsors is $\mathbb{A}^1$-connected, we can point the higher homotopy sheaves $a_{\rm Nis}\pi_j^{\mathbb{A}^1}({\rm B}_{\rm Zar}{\rm GL}_r \simeq {\rm B}_{\et}{\rm GL}_r)$ at the trivial torsor/vector bundle. The lifting problems in~\eqref{eq:stage} are then controlled by Nisnevich cohomology with coefficients in these sheaves. By Theorem~\ref{thm:morel-g}, we have a multitude of choices for base points at which we form the higher homotopy sheaves of ${\rm B}_{\et}G$. This adds an additional complication that one has to deal with when trying to compute these homotopy sheaves. For example, we can point $a_{\rm Nis}\pi_j^{\mathbb{A}^1}({\rm B}_{\et}{\rm PGL}_r)$ at various unramified Brauer classes (see Example~\ref{ex:pgl}) --- at the trivial Brauer class the homotopy sheaves agree with the homotopy sheaves of ${\rm B}_{\rm Nis}{\rm SL}_r$ in a range but this will not be the case at other base points. 
\end{remark}

\begin{remark}\label{rem:morel} We give a brief survey of the state-of-the-art for Morel's conjecture.
\begin{enumerate}
\item Morel made his conjecture in \cite{MField}*{Conjecture 1.12}. This should be contrasted with the main structural results from his book, \cite{MField}*{Theorem 1.9}, which states that the higher homotopy sheaves are strongly/strictly $\mathbb{A}^1$-invariant.
\item There are cases where Morel's conjecture holds rather easily --- $\mathbb{A}^1$-rigid schemes as well as smooth curves over a field.
\item When $G$ is a finite \'etale group scheme, the conjecture holds from \cite{morel:voevodsky}*{Page 137, Corollary 3.2}. Morel has also claimed \cite{MField}*{Remark 1.14} that he has proved that~\eqref{eq:compare-map} is a bijection on perfect fields when the group of irreducible components of $G$ is order prime to $k$ and indicated the possibility of a proof of his conjecture in this case, though no details were given.
\item Work of Choudhury \cite{choudhury} has established Morel's conjecture for grouplike presheaves of $H$-spaces and for principal homogeneous spaces under them. Our arguments in Section~\ref{sec:hoinv} are inspired by some of his arguments.
\item In a different direction, work of Balwe--Hogadi--Sawant \cite{balwe:hogadi:sawant} and Balwe--Sawant \cite{balwe:sawant} has verified Morel's conjecture for smooth, projective surfaces in characteristic zero. 
\end{enumerate}
\end{remark}

\begin{remark}
  The key points in the proofs of Theorems~\ref{thm:cex} and \ref{thm:morel-g} are the Grothendieck--Serre conjecture for $G$-torsors as established in \cites{fedorov:panin,panin} and the $\mathbb{A}^1$-invariance for torsors over fields (established in the characteristic $0$ case by \cite{raghunathan:ramanathan}). The Grothendieck--Serre conjecture allows to reduce questions over irreducible smooth schemes to function fields, where $\mathbb{A}^1$-invariance can be used; this technique was used by Choudhury \cite{choudhury} in his proof of Morel's conjecture for $H$-groups.  
\end{remark}

\subsection{Unramified torsors after Colliot-Th{\'e}l{\`e}ne and Sansuc}
We now explain a consequence of Theorem~\ref{thm:morel-g} to the problem of understanding $G$-torsors on smooth schemes over a field. Let $G$ be a connected reductive group over a field $k$ and let $X$ be a irreducible smooth $k$-scheme with function field $K$. We say that \emph{purity} is satisfied for $G$-torsors over $X$ if the map
  \[
  \op{im}\left({\rm H}^1_{\et}(X,G)\to {\rm H}^1_{\et}(\Spec K,G)\right)\to \bigcap_{x\in X^{(1)}}\op{im}\left({\rm H}^1_{\et}(\Spec \mathcal{O}_{X,x},G)\to {\rm H}^1_{\et}(\Spec K,G)\right)
  \]
  is surjective. An old question of Colliot-Th{\'e}l{\`e}ne and Sansuc \cite{colliot-thelene:sansuc}*{Question 6.4} asks when purity holds. The state-of-the-art is reviewed in Remark~\ref{rem:purity}.

  On the one hand, purity is known for local $X$ in a large number of cases. We observe in Section~\ref{sec:purity} that purity for henselian local $X$ is equivalent to the unramifiedness of the sheaf $\mathcal{H}^1_{\et}(G)$. This establishes the unramifiedness of $a_{\rm Nis}\pi_0^{\mathbb{A}^1}{\rm B}_{\et}G$ in the situation of Theorem~\ref{thm:morel-g}. On the other hand, the answer for general $X$ is expected to be negative. More precisely, we have the following conjecture of Antieau and Williams \cite{antieau:williams}:
  
  \begin{conjecture}[\cite{antieau:williams}*{Conjecture 1.2}]\label{conj:aw} Let $G$ be a non-special, semisimple $k$-group scheme. Then there exists a smooth, affine $k$-scheme for which purity fails.
  \end{conjecture}
  
  \begin{remark} The work of \cite{antieau:williams} proves Conjecture~\ref{conj:aw} in the setting of $k = \mathbb{C}$ and $G = {\rm PGL}_p$ by constructing examples of dimension $2p+2$. Since purity does hold in dimension $\leq 2$, we do not whether or not purity holds in dimension $d = 3, 4, 5$.
  \end{remark}

Using our results on the sheaves of $\mathbb{A}^1$-connected components of classifying spaces, we can bring Conjecture~\ref{conj:aw} to the realm of $\mathbb{A}^1$-homotopy theory in a large range of cases:

\begin{corollary}
  \label{cor:purity-intro}
  Let $G$ be a group over a field $k$ satisfying the hypotheses of Theorem~\ref{thm:cex} and further assume that $G$ satisfies local purity. Then if purity holds for $G$-torsors over $X$ then the sheafification map $\pi_0^{\mathbb{A}^1}{\rm B}_{\et}G(X)\to a_{\rm Nis}\pi_0^{\mathbb{A}^1}{\rm B}_{\et}G(X)$ is surjective. In particular, if $X$ is a smooth $k$-scheme for which the sheafification map $\pi_0^{\mathbb{A}^1}{\rm B}_{\et}G(X)\to a_{\rm Nis}\pi_0^{\mathbb{A}^1}{\rm B}_{\et}G(X)$ is \emph{not} surjective, then we have a counterexample to purity in this setting.
\end{corollary}

It would be interesting to adapt some obstruction-theory methods to identify obstructions to surjectivity of $\pi_0^{\mathbb{A}^1}{\rm B}_{\et}G(X)\to a_{\rm Nis}\pi_0^{\mathbb{A}^1}{\rm B}_{\et}G(X)$ and thus to purity and better understand what properties of schemes $X$ or groups $G$ imply failure of purity.
  

\emph{Structure of the paper:} We first recall a couple of preliminaries in Section~\ref{sec:prelims}. Homotopy invariance for Nisnevich sheafifications is discussed in Section~\ref{sec:hoinv} and unramifiedness for these sheafifications as well as the relation to purity in Section~\ref{sec:purity}. Then we identify the Nisnevich sheaf of \'etale torsors with the $\mathbb{A}^1$-connected components of the classifying spaces in Section~\ref{sec:components} and discuss a couple of examples and consequences in Section~\ref{sec:examples}.

\emph{Acknowledgements:}  We gratefully acknowledge that Girish Kulkarni's stay at Bergische Universit\"at Wuppertal  is supported by a visitor grant from the DFG SPP 1786 ``Homotopy theory and algebraic geometry''. The first author would like to acknowledge the (heavy) influence of Ben Antieau towards the ideas on this project and interesting discussions on \cite{antieau:williams}.

\section{Preliminaries and notation}
\label{sec:prelims}

In this paper, ${\rm Sm}_S$ denote the category of smooth schemes over a base $S$. It is a full subcategory of the category ${\rm EssSm}_S$ of essentially smooth schemes. We will use (very moderately) the language of $\infty$-categories which is, by now, standard in the subject. The $\infty$-category of motivic spaces, as introduced by Morel--Voevodsky \cite{morel:voevodsky} is denoted by ${\rm Spc(S)}$ and we refer to \cites{antieau:elmanto,bachmann:hoyois} for the basics of this construction closer to the language of this paper; in particular a motivic space is an $\mathbb{A}^1$-invariant Nisnevich sheaf of spaces, often denoted by $\mathcal{X}:{\rm Sm}_S^{{\rm op}} \rightarrow \mathcal{S}$.  We will denote by $[\mathcal{X}, \mathcal{Y}]_{\mathbb{A}^1} = \pi_0{\rm Maps}_{\rm Spc(S)}(\mathcal{X}, \mathcal{Y})$ the set of homotopy classes of maps between two motivic spaces. 
The $\infty$-category ${\rm Spc(S)}$ is obtained from the $\infty$-category of presheaves of spaces on $S$ (written as ${\rm PreShv(S)}$) via a combination of two localization endofunctors
\[
L_{\rm Nis}, L_{\rm \mathbb{A}^1}: {\rm PreShv(S)} \rightarrow {\rm PreShv(S)}.
\]
Here, $L_{\rm Nis}$ is the usual Nisnevich localization functor, while $L_{\rm \mathbb{A}^1}$ can be modeled by the ${\rm Sing}$ construction of Suslin, as explained by Morel and Voevodsky in \cite{morel:voevodsky}*{Page 87}; see also \cite{antieau:elmanto}*{Definition 4.2}. In particular if $\mathcal{X}$ is a presheaf of spaces, then
\[
\pi_0(L_{\rm \mathbb{A}^1}\mathcal{X}(T)) = \pi_0({\rm Sing}\mathcal{X}(T)) = {\rm coeq}\left(\pi_0(\mathcal{X}(T \times \mathbb{A}^1)) \rightrightarrows\pi_0(\mathcal{X}(T)\right),
\]
which is a formula we will need in this paper.  One consequence is the so-called unstable $\mathbb{A}^1$-connectivity theorem of Morel-Voevodsky's (\cite{morel:voevodsky}*{Corollary 3.22}, \cite{antieau:elmanto}*{Corollary 4.30}); we write $\pi_0^{\mathbb{A}^1}(\mathcal{X})$ (resp. $\pi_0(\mathcal{X})$) for the presheaf $U\mapsto [U,\mathcal{X}]_{\mathbb{A}^1}$ (resp. $U \mapsto [U, \mathcal{X}]$) and we write $a_{\rm Nis}\pi_0^{\mathbb{A}^1}(\mathcal{X})$ (resp. $a_{\rm Nis}\pi_0(\mathcal{X})$) for the sheafification (as presheaves of sets). Then, for any presheaf of spaces $\mathcal{X}$, the map
\[
a_{\rm Nis}\pi_0(\mathcal{X}) \rightarrow a_{\rm Nis}\pi_0^{\mathbb{A}^1}(\mathcal{X})
\] 
is an epimorphism of Nisnevich sheaf of sets. 

Many arguments in this paper require passing to generic points or to stalks.
\begin{notation}
  Let $\mathcal{C}$ be an $\infty$-category with colimits. We adopt the following conventions: for a presheaf $\mathcal{F}\colon {\rm Sm}_S^{\rm op}\to\mathcal{C}$ we get a presheaf ${\rm EssSm}_S^{\rm op}\to\mathcal{C}$ by left Kan extension which we abusively also call $\mathcal{F}$. In particular any motivic space canonically extends to a functor out of ${\rm EssSm}_S$.
\end{notation}

\begin{remark} We have a fully faithful immersion ${\rm EssSm}_S \hookrightarrow {\rm Pro}({\rm Sm}_S)$, where the ${\rm EssSm}_S$ is the subcategory of pro-objects with affine transition maps. Since any presheaf on ${\rm Sm}_S$ extends uniquely to one on ${\rm Pro}({\rm Sm}_S)$ (via Left Kan extension), it uniquely determines a presheaf on ${\rm EssSm}_S$ by restriction.
\end{remark}

\subsection{Torsors}
We recall that for an algebraic group $G$, a $G$-torsor $Y$ over a scheme $X$ is a $X$-scheme $Y$ equipped with a $G$-action such that $$G\times_X Y \to Y\times_X Y $$  is an isomorphism. A torsor is called rationally trivial if there is an open subset $U$ of $X$ such that $Y\times_X U$ is trivial. It follows from the Seshadri's result \cite{seshadri} that all rationally trivial torsors over smooth schemes are locally trivial in Nisnevich topology. In general torsors are only locally trivial in the \'etale topology.

\subsection{Nisnevich sheaves of sets}
While the presheaves of interest in this paper come as presheaves of spaces, we will be mostly interested in their homotopy sheaves, in particular their $\pi_0$. We will now recall some formalism about Nisnevich sheafification of sets. Denoting Nisnevich sheafification by $a_{\rm Nis}$, one of the ways to get this sheafification is by using Godement resolution, see \cite{godement} for details. For a presheaf $\mathcal{F}\colon {\rm Sm}_S^{\rm op}\to {\rm Set}_\ast$, the Godement resolution $\mathcal{F}(-) \to \mathcal{G}F(-)$ is defined as follows:
Let $$ G^0\mathcal{F}(U)= \prod_{u\in U}\mathcal{F}(\mathcal{O}^{\rm h}_{U,u})$$ 
and $G^n\mathcal{F}=G\circ \cdots \circ G \mathcal{F}$, the $n+1$ fold composition of $G$. Now define $\mathcal{G}F(U)=\op{Tot}\mathcal{G}^{\bullet}F(U)$, then $\mathcal{F}\to \mathcal{G}F$
is a flasque resolution of $\mathcal{F}$ on $({\rm Sm}_S^{\rm op})_{\rm Nis}$. So a section of the sheafification $a_{\rm Nis}\mathcal{F}$ is given by a collection of sections of $\mathcal{F}$ over the henselizations $\mathcal{O}_{X,x}^{\rm h}$ of the local rings of points $x\in X$, subject to compatibility conditions. From this formula, we see that a section of the sheafification $a_{\rm Nis}\mathcal{F}$ is given by a collection of sections of $\mathcal{F}$ over the henselizations $\mathcal{O}_{X,x}^{\rm h}$ of the local rings of points $x\in X$, with compatibilities 

Let $\mathcal{F}(-)={\rm H}^1_{\et}(-,G)$ denote the presheaf associating to a smooth scheme the set of isomorphism classes of $G$-torsor. The  Nisnevich sheafification of this presheaf is denoted by $\mathcal{H}^1_{\et}(G)$.

The following is a list of properties of presheaves which we will use throughout the paper. 

\begin{definition}
  Let $F$ be a field and $\mathcal{F}\colon{\rm EssSm}_F^{\rm op}\to {\rm Set}_*$ be a presheaf of pointed sets 
  \begin{enumerate}
  \item The presheaf $\mathcal{F}$ is called \emph{finitary} if it converts cofiltered limits to filtered colimits, i.e., it is left Kan extended from ${\rm Sm}_F$. 
  \item The presheaf $\mathcal{F}$ is called \emph{homotopy invariant} if for any $X\in{\rm EssSm}_F$ we have an induced isomorphism $\mathcal{F}(X)\xrightarrow{\cong}\mathcal{F}(X\times\mathbb{A}^1)$.
  \item The presheaf $\mathcal{F}$ has the (strong) \emph{Grothendieck--Serre property} if for any regular local $F$-algebra $R$ with fraction field $K$, the map $\mathcal{F}(\Spec R)\to \mathcal{F}(\Spec K)$ has trivial kernel (is injective). 
  \item The presheaf $\mathcal{F}$ is \emph{Nisnevich lexcisive} if for an Nisnevich distinguished square
    \[
  \xymatrix{
   W \ar[d] \ar[r] & X \ar[d] \\
    U \ar[r] & Y,
  }
  \]
  the map
  \[
  \mathcal{F}(Y)  \rightarrow \mathcal{F}(X) \times_{\mathcal{F}({W})} \mathcal{F}(U)
  \]
is a surjective map of sets.
\item The presheaf $\mathcal{F}$ is \emph{Nisnevich excisive} if it takes a Nisnevich distinguished square as above to a pullback square of pointed sets. 
  \end{enumerate}
\end{definition}

\section{Homotopy invariance for Nisnevich sheafifications}
\label{sec:hoinv}

In the following section we discuss the homotopy invariance of Nisnevich sheafifications, with a particular view toward $\mathcal{H}^1_{\et}(G)$. Essentially, any finitary presheaf which satisfies a strong version of the Grothendieck--Serre conjecture (i.e., that restricting sections from local rings to their function fields is injective) and $\mathbb{A}^1$-invariance over fields has a homotopy invariant Nisnevich sheafification. A similar method has already been used by Choudhury \cite{choudhury} for proving that $a_{\rm Nis}\pi_0^{\mathbb{A}^1}(\mathcal{X})$ is homotopy invariant for H-groups and principal homogeneous spaces under H-groups.

We first note that the strong Grothendieck--Serre property (which is about local rings) actually implies that sections over smooth schemes are detected on the function fields:

\begin{proposition}
  \label{prop:lnis-unramified-1}
  Let $\mathcal{F}$ be a finitary presheaf on ${\rm EssSm}_F$ which satisfies the strong Grothendieck--Serre property. Then for every irreducible essentially smooth $F$-scheme $X$ with function field $K$ the restriction map $a_{\rm Nis} \mathcal{F}(X)\to a_{\rm Nis}\mathcal{F}(\Spec K)\cong \mathcal{F}(\Spec K)$ to the generic point is injective. 
\end{proposition}

\begin{proof}
  From the Godement description of the sheafification, cf. Section~\ref{sec:prelims}, we have the following commutative square in which the left vertical map is injective:
  \[
  \xymatrix{
    a_{\rm Nis}\mathcal{F}(X) \ar[r] \ar@{^(->}[d] & \mathcal{F}(\Spec K) \ar[d] \\
    \prod_{x\in X}\mathcal{F}(\Spec\mathcal{O}_{X,x}^{\rm h}) \ar[r] & \prod_{x\in X}\mathcal{F}(\Spec\op{Frac}(\mathcal{O}_{X,x}^{\rm h}))
  }
  \]
  It therefore suffices to check injectivity of the bottom map. But that follows since $\mathcal{F}$ has the strong Grothendieck--Serre property.
\end{proof}

\begin{remark}
  Actually, the strong Grothendieck--Serre property implies the injectivity already for the Zariski sheafification. For the Nisnevich result, a strong Grothendieck--Serre property for henselian local rings would suffice.
\end{remark}

In the following, we provide versions of \cite{choudhury}*{Lemma 3.3, Theorem 3.1}. These say that sections of Nisnevich sheafifications of finitary presheaves satisfying Nisnevich descent are induced from presheaf sections defined at all codimension 1 points. These results will be needed for reducing the general homotopy invariance for the Nisnevich sheafification to the invariance for $\mathbb{A}^1$ over fields for the presheaf, as input in Proposition~\ref{prop:lnis-hoinv}. They also play some role in the discussion of unramifiedness and purity later.


\begin{proposition}
  \label{prop:lnis-dvr}
  Let $\mathcal{F}$ be a finitary presheaf on ${\rm EssSm}_F$ satisfying Nisnevich lexcision and let $R$ be an essentially smooth discrete valuation ring. Then the sheafification map is surjective
  \[
  \mathcal{F}(R)\twoheadrightarrow a_{\rm Nis}\mathcal{F}(R).
  \]
\end{proposition}

\begin{proof}
  Denote by $K=\op{Frac}(R)$ the fraction field of $R$ and by $R^{\rm h}$ the henselization of $R$ at the maximal ideal. Let $\sigma\in a_{\rm Nis}\mathcal{F}(R)$ be a section. For the restricted section $\sigma|_{R^{\rm h}}\in a_{\rm Nis}\mathcal{F}(R^{\rm h})\cong\mathcal{F}(R^{\rm h})$ there exists a Nisnevich neighbourhood $p\colon W\to \Spec R$ and a section $\sigma'\in \mathcal{F}(W)$ mapping to $\sigma_{R^{\rm h}}$. Since fields are Nisnevich local, the restrictions of $\sigma'$ and $\sigma$ to $\op{Frac}(W)$ coincide. By Nisnevich lexcision, there exists a section $\tau\in\mathcal{F}(R)$, whose restriction to $W$ agrees with $\sigma'$ and whose restriction to $\op{Frac}(R)$ agrees with the corresponding restriction of $\sigma$. This proves the surjectivity.
\end{proof}

\begin{proposition}
  \label{prop:lnis-dim1}
  Let $\mathcal{F}$ be a finitary presheaf on ${\rm EssSm}_F$ which satisfies Nisnevich lexcision. Then for any essentially smooth $F$-scheme $X$ of dimension $\leq 1$, the sheafification map $\mathcal{F}(X)\to a_{\rm Nis}\mathcal{F}(X)$ is surjective. 
\end{proposition}

\begin{proof}  
The proof of \cite{choudhury}*{Theorem 3.1} goes through almost verbatim. Consider an essentially smooth scheme $X\in{\rm EssSm}_F$ of dimension 1 with function field $K$. We can reduce to the case that $X$ is connected. Since $\mathcal{F}$ is finitary, we can reduce to the case that $X$ is noetherian.

Let $\sigma\in a_{\rm Nis}\mathcal{F}(X)$. By the Godement description of Nisnevich sheafification, cf. Section~\ref{sec:prelims}, this means we are given $\sigma_x\in a_{\rm Nis}\mathcal{F}(\Spec\mathcal{O}_{X,x})$ for any codimension 1 point $x\in X$ and a generic section $\sigma_\eta\in a_{\rm Nis}\mathcal{F}(\Spec K)=\mathcal{F}(\Spec K)$ and the restrictions of $\sigma_x$ to $\Spec K$ agree with $\sigma_\eta$. By Proposition~\ref{prop:lnis-dvr}, we have for every $x\in X^{(1)}$ a section $\sigma'_x\in\mathcal{F}(\Spec\mathcal{O}_{X,x})$ mapping to $\sigma_x$ under the sheafification map.

  Since $\mathcal{F}$ is finitary, the section $\sigma_\eta$ extends to a section $\sigma_U\in\mathcal{F}(U)$ for an open subset $U\subseteq X$ whose complement then consists of finitely many points. We want to show that $\sigma_U$ extends over those finitely many points. For any such point $x\in X\setminus U$ there exists an open neighbourhood $V\subseteq X$ of $x$ and a section $\tau_x\in\mathcal{F}(V)$ which induces $\sigma_x'\in\mathcal{F}(\mathcal{O}_{X,x})$. Possibly shrinking $V$, the section $\tau_x$ coincides with $\sigma_U$ on $U\cap V$ because the restriction of $\sigma_x'$ to the generic point coincides with $\sigma_\eta$. By Nisnevich excision applied to the covering of $(X\setminus U)\cup\{x\}$ by $U$ and $V$ we get a section $\tau\in\mathcal{F}((X\setminus U)\cup\{x\})$ inducing $\sigma_U$ and $\tau_x$, thus extending $\sigma_U$ over the point $x$. 
\end{proof}


The following results is in some way already contained in \cite{choudhury}. Choudhury's proof of Morel's conjecture for $H$-groups is based on the fact that for sheaves of groups the weak and strong Grothendieck--Serre property agree and the weak version of the Grothendieck--Serre property follows since $\pi_0^{\mathbb{A}^1}(\mathcal{X})$ is a finitary homotopy-invariant presheaf satisfying Nisnevich excision.

\begin{proposition}
  \label{prop:lnis-hoinv}
  Let $\mathcal{F}$ be a finitary and Nisnevich lexcisive presheaf on ${\rm EssSm}_F$ which satisfies the following:
  \begin{enumerate}[(a)]
  \item $\mathcal{F}$ has the strong Grothendieck--Serre property.
  \item For any finitely generated field extension $L/F$ the projection $\mathbb{A}^1_L\to \Spec L$ induces a bijection $a_{\rm Nis}\mathcal{F}(\Spec L)\xrightarrow{\cong}a_{\rm Nis}\mathcal{F}(\mathbb{A}^1_L)$.
  \end{enumerate}
  Then $a_{\rm Nis}\mathcal{F}$ is homotopy invariant.
\end{proposition}

\begin{proof}
  We want to show that for any smooth scheme $X$ the projection $\op{pr}_1\colon X\times\mathbb{A}^1\to X$ induces a bijection
  \[
  a_{\rm Nis}\mathcal{F}(X)\xrightarrow{\cong} a_{\rm Nis}\mathcal{F}(X\times\mathbb{A}^1).
  \]
  This map is always injective since the composition $X\xrightarrow{0}X\times\mathbb{A}^1\xrightarrow{\op{pr}_1} X$ is the identity. For the proof of surjectivity, let $\sigma\in a_{\rm Nis}\mathcal{F}(X\times\mathbb{A}^1)$ be a section. Pullback along the composition $X\times\mathbb{A}^1\xrightarrow{\op{pr}_1} X\xrightarrow{0}X\times\mathbb{A}^1$ produces a constant section $\tau$, and we want to show that $\sigma=\tau$ in $a_{\rm Nis}\mathcal{F}(X\times\mathbb{A}^1)$. 

  Let $X$ be an irreducible smooth scheme $X$, denote by $K$ the function field of $X$, and choose a coordinate $T$ on $\mathbb{A}^1$. By Assumption (a) and Proposition~\ref{prop:lnis-unramified-1}, the composition
  \[
  a_{\rm Nis}\mathcal{F}(X\times\mathbb{A}^1)\to a_{\rm Nis}\mathcal{F}(\Spec K[T])\to a_{\rm Nis}\mathcal{F}(\Spec K(T))
  \]
  is injective. Thus the first map is injective, and it suffices to show that the restrictions of $\sigma$ and $\tau$ agree in $a_{\rm Nis}\mathcal{F}(\Spec K[T])$. From the following commutative diagram
  \[
  \xymatrix{
    \Spec(K[T])\ar[d]_{\iota[T]} \ar[r]^{T\mapsto 0} & \Spec(K) \ar[d]_\iota \ar[r]^c & \Spec(K[T]) \ar[d]^{\iota[T]} \\
    X\times \mathbb{A}^1 \ar[r]_{\op{pr}_1} & X \ar[r]_0 & X\times\mathbb{A}^1
  }
  \]
  where $c$ is the inclusion of constants and $\iota\colon \Spec(K)\hookrightarrow X$ the inclusion of the generic point, we see that restricting $\tau$ along $\iota[T]$ is the same as making the restriction of $\sigma$ along $\iota[T]$ constant by pullback along $\Spec(K[T])\xrightarrow{T\mapsto 0} \Spec(K)\xrightarrow{c} \Spec(K[T])$. 

  Denote by $\sigma_\eta$ and $\tau_\eta$ the restrictions of $\sigma$ and $\tau$ to $\Spec K[T]$. By Assumption (b), all the maps in the composition
  \[
  a_{\rm Nis}\mathcal{F}(\Spec K)\xrightarrow{\op{pr}_1^\ast} a_{\rm Nis}\mathcal{F}(\Spec K[T]) \xrightarrow{0^\ast} a_{\rm Nis}\mathcal{F}(\Spec K)
  \]
  are bijections, the composition being the identity. Therefore, $\sigma_\eta=\tau_\eta$ in $a_{\rm Nis}\mathcal{F}(\mathbb{A}^1_K)$ because by construction their images under $0^\ast$ agree. This finishes the proof.
\end{proof}

Note that since we can glue torsors in the Nisnevich topology $H^1_{\et}(\--,G)$ satisfies Nisnevich lexcision. Now we now want to apply this result to prove homotopy invariance for the Nisnevich sheaf $\mathcal{H}^1_{\et}(G)$. To do that, we recall some results on $\mathbb{A}^1$-invariance of \'etale torsors over fields. The most general result we know of is the following, from \cite{raghunathan:ramanathan}.

\begin{theorem}
  \label{thm:raghunathan-ramanathan}
  Let $F$ be a field with separable closure $F^{\rm sep}$ and let $G$ be a connected reductive group. Then any $G$-torsor over $\mathbb{A}^1_F$ which becomes trivial over $\mathbb{A}^1_{F^{\rm sep}}$ is extended from $F$, i.e.,  pullback induces a bijection
  \[
    {\rm H}^1_{\et}(\Spec F,G)\to\ker\left({\rm H}^1_{\et}(\mathbb{A}^1_F,G)\to  {\rm H}^1_{\et}(\mathbb{A}^1_{F^{\rm sep}},G)\right).
    \]
\end{theorem}

In particular, for a field $F$ of characteristic $0$, \'etale torsors satisfy $\mathbb{A}^1$-invariance over all extension fields $L/F$ in the sense that the pullback map ${\rm H}^1_{\et}(\Spec L,G)\to{\rm H}^1_{\et}(\mathbb{A}^1_L,G)$ is a bijection. In this case, we get homotopy invariance for $\mathcal{H}^1_{\et}(G)$  for any connected reductive group $G$. 

\begin{proposition}
  \label{prop:hoinv-h1etg}
  Let $F$ be a field of characteristic $0$ and $G$ be a connected reductive group over $F$. Then the sheaf $\mathcal{H}^1_{\et}(G)$ is homotopy invariant.
\end{proposition}

\begin{proof}
We use Proposition~\ref{prop:lnis-hoinv} with the presheaf $U\mapsto {\rm H}^1_{\et}(U,G)$. The strong Grothendieck--Serre property follows from \cite{fedorov:panin}. For the $\mathbb{A}^1$-invariance over extension fields $L/F$, we note that by Proposition~\ref{prop:lnis-dim1} the map ${\rm H}^1_{\et}(\mathbb{A}^1_L,G)\to \mathcal{H}^1_{\et}(G)(\mathbb{A}^1_L)$ is surjective, i.e., any section of $\mathcal{H}^1_{\et}(G)$ is actually induced from a $G$-torsor over $\mathbb{A}^1_L$. After base change to the separable closure $L^{\rm sep}$, the resulting $G$-torsor over $\mathbb{A}^1_{L^{\rm sep}}$ becomes trivial, by a theorem of Steinberg \cite{steinberg}*{Theorem 1.9}. Thus we can apply Theorem~\ref{thm:raghunathan-ramanathan} to see that the $G$-torsor over $\mathbb{A}^1_L$ must be constant, finishing the proof.
\end{proof}

\begin{remark}
  This generalizes the observation that for Parimala's non-extended quadratic form over $\mathbb{A}^2_{\mathbb{R}}$, there is an open subset of $\mathbb{A}^2_{\mathbb{R}}$ where the torsor is constant and extended from the anisotropic form over $\mathbb{R}$, cf. \cite{lam:book}*{Lemma VII.4.14 resp. Lemma VII.4.16}.
\end{remark}

We note two other cases of interest over fields of positive characteristic where we get homotopy invariance for $\mathcal{H}^1_{\et}(G)$.

\begin{proposition}
  \label{prop:hoinv-on}
  Let $F$ be a field of characteristic $\neq 2$. Then the sheaf $\mathcal{H}^1_{\et}({\rm O}(n))$ is homotopy invariant. 
\end{proposition}

\begin{proof}
  As in the proof of Proposition~\ref{prop:hoinv-h1etg}, we apply Proposition~\ref{prop:lnis-hoinv} to the presheaf $U\mapsto {\rm H}^1_{\et}(U,G)$ using the strong Grothendieck--Serre property from \cite{fedorov:panin} (for infinite fields) and \cite{panin} (for finite fields). The $\mathbb{A}^1$-invariance for ${\rm O}(n)$-torsors over fields of characteristic $\neq 2$ is Harder's theorem, cf.~\cite{knebusch}*{Theorem 13.4.3} resp. \cite{lam:book}*{Theorem VII.3.13}.
\end{proof}

\begin{proposition}
  \label{prop:hoinv-pgln}
  Let $F$ be a field of characteristic $p$. If $p\nmid n$, then the sheaf $\mathcal{H}^1_{\et}({\rm PGL}_n)$ is homotopy invariant. 
\end{proposition}

\begin{proof}
  As before, we need only deal with the $\mathbb{A}^1$-invariance over fields. By Theorem~\ref{thm:raghunathan-ramanathan}, it suffices to show that for $L$ a separably closed extension field of $F$ all ${\rm PGL}_n$-torsors over $\mathbb{A}^1_L$ are trivial. Consider the exact sequence (of pointed sets)
  \[
    {\rm H}^1_{\et}(\mathbb{A}^1_L,{\rm GL}_n)\to {\rm H}^1_{\et}(\mathbb{A}^1_L,{\rm PGL}_n)\to {\rm H}^2_{\et}(\mathbb{A}^1_L,\mathbb{G}_{\rm m})
  \]
  associated to the extension $\mathbb{G}_{\rm m}\to {\rm GL}_n\to {\rm PGL}_n$. Here ${\rm H}^2_{\et}(\mathbb{A}^1_L,\mathbb{G}_{\rm m})$ is the cohomological Brauer group since $\mathbb{A}^1_L$ is regular (by a theorem of Grothendieck) which in turn is identified with the Brauer group ${\rm Br}(\mathbb{A}^1_L)$ (by a theorem of Gabber). Exactness for the sequence of pointed sets then means that any ${\rm PGL}_n$-torsor whose Brauer class is trivial comes from a vector bundle (and therefore has to be trivial). By the Auslander--Goldman theorem (together with our assumption that $L$ is separably closed), ${\rm Br}(\mathbb{A}^1_L)$ is a $p$-torsion group. Since the boundary map factors as
  \[
    {\rm H}^1_{\et}(\mathbb{A}^1_L,{\rm PGL}_n)\to {\rm H}^2_{\et}(\mathbb{A}^1_L,\mu_n)\to {\rm H}^2_{\et}(\mathbb{A}^1_L,\mathbb{G}_{\rm m})
  \]
  and ${\rm H}^2_{\et}(\mathbb{A}^1_L,\mu_n)$  is an $n$-torsion group, our assumption $p\nmid n$ implies that ${\rm PGL}_n$-torsors over $\mathbb{A}^1_L$ have trivial Brauer classes. The exact cohomology sequence now implies that every ${\rm PGL}_n$-torsor over $\mathbb{A}^1_L$ is trivial, as required.
\end{proof}

\begin{remark}
  In \cite{knus:ojanguren:saltman}, there are examples of ${\rm PGL}_n$-torsors over the affine line $\mathbb{A}^1_L$ for separably closed but non-algebraically closed fields $L$ which are not extended. These are related to non-trivial $p$-coverings of $\mathbb{A}^1_L$, where $p=\op{char}(L)$ (in particular these torsors are not rationally trivial). This means that the divisibility condition in Proposition~\ref{prop:hoinv-pgln} is necessary and homotopy invariance generally fails for ${\rm PGL}_n$-torsors if the characteristic of the base field divides $n$.
\end{remark}

\begin{remark}
  In case anyone is interested, it is also possible to show that $\mathcal{H}^1_{\et}({\rm G}_2)$ is homotopy invariant over base fields of characteristic $\neq 2$ and that $\mathcal{H}^1_{\et}({\rm F}_4)$ is homotopy invariant over base fields of characteristic $\neq 2,3$.

  For this,  we can reduce to the case of a separably closed field $L$ of characteristic $\neq 2$ as in Proposition~\ref{prop:hoinv-pgln}. For a $G$-torsor $\mathcal{T}$ over $\mathbb{A}^1_L$, the restriction to the function field $L(T)$ is trivial if and only if the relevant cohomological invariants are trivial, cf. \cite{serre}. For $G={\rm G}_2$, the classifying cohomological invariant is the class of the norm form in ${\rm H}^3_{\et}(-,\mu_2)$, and for $G={\rm F}_4$, the cohomological invariants are Pfister forms $f_3\in{\rm H}^3_{\et}(-,\mu_2)$ and $f_5\in {\rm H}^5_{\et}(-,\mu_2)$ and the Rost invariant $g_3\in{\rm H}^3_{\et}(-,\mu_3)$. Since $L$ is separably closed, all these \'etale cohomology groups vanish for $L(T)$ which has \'etale cohomological dimension 1. Therefore, the restriction of the $G$-torsor $\mathcal{T}$ to the generic point of $\mathbb{A}^1_L$ is trivial because it has trivial cohomological invariants. Since rationally trivial torsors satisfy homotopy invariance, $\mathcal{T}$ has to be trivial.

  This argument via rationally trivial torsors could also be helpful to establish homotopy invariance of $\mathcal{H}^1_{\et}(G)$ for other cases of connected reductive groups $G$ in positive characteristic.
\end{remark}

\begin{remark}
  The most general result one could expect here is that $\mathbb{A}^1$-invariance (and thus homotopy invariance generally) holds for a reductive group $G$ over a field $F$ of characteristic $p$ if both the group $\pi_0(G)$ of connected components of $G$ and the Chevalley fundamental group $\Pi_1(G)$ of $G$ have orders prime to the characteristic of the base field, and $p$ is not a torsion prime for any of the almost simple components of $G$. However, results in this generality seem not known at this point.
\end{remark}


\section{Purity for torsors}
\label{sec:purity}

In this section we show that unramifiedness of the sheaf $\mathcal{H}^1_{\et}(G)$ is equivalent to Nisnevich-local purity for $G$-torsors. In situations where local purity results are known, this allows to compute $\mathcal{H}^1_{\et}(G)$ and to reinterpret purity for torsors as surjectivity of the sheafification map ${\rm H}^1_{\et}(X,G)\to \mathcal{H}^1_{\et}(G)(X)$ for smooth schemes $X$. It will also provide a relation between global purity questions and motivic homotopy, cf. Section~\ref{sec:examples}.

\begin{definition}
  \label{def:purity}
  Let $G$ be a connected reductive group over a field $F$ and let $X$ be an irreducible smooth $F$-scheme with function field $K$. We say that \emph{purity} is satisfied for $G$-torsors over $X$ if the map
  \[
  \op{im}\left({\rm H}^1_{\et}(X,G)\to {\rm H}^1_{\et}(\Spec K,G)\right)\to \bigcap_{x\in X^{(1)}}\op{im}\left({\rm H}^1_{\et}(\Spec \mathcal{O}_{X,x},G)\to {\rm H}^1_{\et}(\Spec K,G)\right)
  \]
  is surjective.
\end{definition}

\begin{remark}
  If $G$ is a special group in the sense of Serre, i.e., $G$-torsors over reduced $F$-varieties are Zariski-locally trivial then ${\rm H}^1_{\et}(\Spec K,G) = \{ \ast \}$ so that Definition~\ref{def:purity} is trivially satisfied.
\end{remark}

\begin{remark}\label{rem:purity}
  The question of purity for torsors was formulated by Colliot-Th{\'e}l{\`e}ne and Sansuc in \cite{colliot-thelene:sansuc}*{Question 6.4}. They also proved a purity theorem for the case of integral regular schemes $X$ of dimension $2$ and reductive groups $G$, cf. \cite{colliot-thelene:sansuc}*{Theorem 6.13 resp. Corollary 6.14}. Local purity, i.e., purity for regular local rings, is known in many cases of connected reductive groups over characteristic $0$ fields, e.g. most of the classical groups \cite{panin:purity}. Local purity for ${\rm PGL}_n$-torsors whose Brauer class has invertible exponent has been proved in \cite{antieau:williams}*{Theorem 3.10}. Local purity for the orthogonal groups in characteristic $\neq 2$ has been proved in \cite{panin:pimenov}.  Local purity for ${\rm G}_2$-torsors  over fields of characteristic $\neq 2$ has been proved in \cite{chernousov:panin} (in combination with \cite{panin:pimenov}*{Remark 3.2}). Local purity for $F_4$ torsors with trivial $g_3$ invariant in characteristic $\neq 2$ has been proved in \cite{chernousov:paninA}.

  However, it seems at this point that local purity hasn't been proved in general (and uniformly, not based on case-by-case analysis), not even for henselian regular local rings. Most of the proofs seem to work via stabilization in infinite series (for the classical groups) combined with the additional benefits on a group structure for the stable groups, or via cohomological invariants (for ${\rm G}_2$ and ${\rm F}_4$).
\end{remark}

For the sake of completeness, let us recall the following notion from \cite{MField}*{Definition 2.1}.

\begin{definition}\label{def:unramified} Let $F$ be a field. A finitary presheaf $\mathcal{G}: {\rm EssSm}^{\rm op}_F \rightarrow {\rm Sets}$ is said to be \emph{unramified} if 
\begin{enumerate}
\item if $X \in {\rm Sm}_F$ is a union of irreducible components $X_{\alpha}$ then $\mathcal{G}(X) \rightarrow \prod_{\alpha} \mathcal{G}(X_{\alpha})$ is a bijection,
\item for any $X \in {\rm Sm}_F$ and any dense open $U \subset X$, the map $\mathcal{G}(X) \rightarrow \mathcal{G}(U)$ is injective and,
\item if $X \in {\rm Sm}_F$ is furthermore irreducible and $F(X)$ is its function field, then the injection (of subsets of $\mathcal{G}(F(X))$
\[
\mathcal{G}(X) \rightarrow \cap_{x \in X^{(1)}} \mathcal{G}(\mathcal{O}_{X,x})
\]
is a bijection.
\end{enumerate}
\end{definition}

We note that any unramified presheaf is, in fact, a Zariski sheaf \cite{MField}*{Remark 2.2}. 

\begin{lemma}\label{lem:unram-gs}   Let $\mathcal{G}$ be a finitary presheaf on ${\rm EssSm}_F$ which is furthermore a Nisnevich sheaf and satisfies (1) and (2) of Definition~\ref{def:unramified}. Then $\mathcal{G}$ is unramified if and only if it satisfies:
\begin{enumerate}
\item[(3')] for every $X \in {\rm Sm}_F$ and any $x \in X$, writing $Y = \Spec \mathcal{O}^{\rm h}_{X,x}$, the injection (of subsets of $\mathcal{G}({\rm Frac}(\mathcal{O}^{\rm h}_{X,x}))$):
\[
\mathcal{G}(Y) \rightarrow \bigcap_{y \in Y^{(1)}}\mathcal{G}(\mathcal{O}_{Y,y})
\]
is a bijection.
\end{enumerate}
\end{lemma}

\begin{proof} If $\mathcal{G}$ is unramified, then $(3')$ is clearly satisfied. We now assume $(3')$. Let $X \in {\rm Sm}_F$ be irreducible and consider the injective map:
\[
\mathcal{G}(X) \to \bigcap_{x\in X^{(1)}} \mathcal{G}(\mathcal{O}_{X,x}),
\]
appearing in condition (3) of Definition~\ref{def:unramified}. We need to prove that this map is surjective. Noting that $\mathcal{G}$ is a Nisnevich sheaf it suffices to produce compatible sections of $\mathcal{G}(\mathcal{O}^{\rm h}_{X,y})$ as $y$ ranges across all points of $y \in X$. 

To this end, fix $y \in X$ and let $Y =  \Spec \mathcal{O}^{\rm h}_{X,y}$. An element $\alpha$ of $\bigcap_{x\in X^{(1)}} \mathcal{G}(\mathcal{O}_{X,x})$ is a section of $\mathcal{G}$ over the function field $K$ of $X$ which extends over all codimension $1$ points of $X$. Now, for every $\mathfrak{p} \in (\Spec \mathcal{O}^{\rm h}_{X,y})^{(1)}$, there exists an $\mathfrak{q} \in (\Spec \mathcal{O}_{X,y})^{(1)}$ such that the following diagram of local rings and local homomorphisms commutes:
  \[
  \xymatrix{
  \mathcal{O}_{X,y} \ \ar[r] \ar[d] & \mathcal{O}^{\rm h}_{X,y} \ \ar[d] \\
   (\mathcal{O}_{X,y})_{\mathfrak{q}}  \ar[r] & (\mathcal{O}^{\rm h}_{X,y})_{\mathfrak{p}}.
  }
  \]
  Since codimension one points of $\Spec \mathcal{O}_{X,x}$ are just codimension one points of $X$ in the closure of $x$, we obtain a section of $\mathcal{G}((\mathcal{O}^{\rm h}_{Y,y})_{\mathfrak{p}})$ from $\alpha$. By construction, the restriction of this section to $\op{Frac}(\mathcal{O}^{\rm h}_{Y,y})$ agrees with the restriction of $\alpha$ along the inclusion $K\hookrightarrow \op{Frac}(\mathcal{O}^{\rm h}_{Y,y})$. Varying $\mathfrak{p}\in Y$, this defines a section of $ \bigcap_{\mathfrak{p} \in Y^{(1)}}\mathcal{G}(\mathcal{O}_{Y,\mathfrak{p}})$ and thus a section of $\mathcal{G}(Y)$ by (3') whose restriction to $\op{Frac}(\mathcal{O}_{X,y}^{\rm h})$ agrees with the restriction of $\alpha$ along the inclusion $K\hookrightarrow \op{Frac}(\mathcal{O}^{\rm h}_{Y,y})$. This concludes the proof.
\end{proof}

The resolution of the Grothendieck-Serre conjecture, due to Fedorov and Panin \cite{fedorov:panin}, lets us reformulate local henselian purity in terms of unramifiedness.

\begin{lemma}
  \label{lem:purity-unramified}
  Let $F$ be a field, and let $G$ be a reductive group over $F$. Purity for henselizations $R=\mathcal{O}_{X,x}^{\rm h}$ of local rings of smooth $F$-schemes  is equivalent to unramifiedness of $\mathcal{H}^1_{\et}(G)$. 
\end{lemma}

\begin{proof}
  We first note that for a discrete valuation ring $R$ with fraction field $K$ the sheafification map induces a bijection
  \[
  \op{im}\left({\rm H}^1_{\et}(\Spec R, G)\to {\rm H}^1_{\et}(\Spec K,G)\right)\xrightarrow{\cong} \mathcal{H}^1_{\et}(G)(\Spec R).
  \]
  The injectivity follows from the strong Grothendieck--Serre property \cite{fedorov:panin} and the surjectivity follows from Proposition~\ref{prop:lnis-dvr}. As a consequence, for an irreducible essentially smooth $F$-scheme $X$ with function field $K$, we get an induced equality of subsets of ${\rm H}^1_{\et}(\Spec K,G)$:
  \[
  \bigcap_{x\in X^{(1)}}\op{im}\left({\rm H}^1_{\et}(\Spec \mathcal{O}_{X,x},G)\to {\rm H}^1_{\et}(\Spec K,G)\right) = \bigcap_{x\in X^{(1)}}\mathcal{H}^1_{\et}(G)(\mathcal{O}_{X,x})
  \]
  Now let $R=\mathcal{O}_{X,x}^{\rm h}$ be the henselization of an irreducible smooth $F$-scheme $X$ at a point $x\in X$. Denote $Y=\Spec R$ and $K={\rm Frac}(R)$ and consider the diagram
  \[
  \xymatrix{
    {\rm H}^1_{\et}(Y,G) \ar[r] \ar[d] & \bigcap_{y\in Y^{(1)}}\op{im}\left({\rm H}^1_{\et}(\Spec \mathcal{O}_{Y,y},G)\to {\rm H}^1_{\et}(\Spec K,G)\right) \ar[d] \\
    \mathcal{H}^1_{\et}(G)(Y) \ar[r] & \bigcap_{y\in Y^{(1)}}\mathcal{H}^1_{\et}(G)(\mathcal{O}_{Y,y})
  }
  \]
  The diagram is commutative since the vertical sheafification maps commute with the horizontal restrictions to ${\rm H}^1_{\et}(\Spec K,G)\cong\mathcal{H}^1_{\et}(G)(\Spec K)$. The right vertical arrow is a bijection by what we noted at the beginning of the proof, and the left vertical arrow is a bijection by the assumptions on $R$. Therefore, the upper horizontal map is a bijection if and only if the lower horizontal map is. But (by definition) the upper horizontal arrow is a surjection if and only if purity holds.

 Now, Proposition~\ref{prop:lnis-unramified-1} ensures that $\mathcal{H}^1_{\et}(G)(X) \rightarrow \mathcal{H}^1_{\et}(G)(K)$ is injective. Lemma~\ref{lem:unram-gs} implies that $\mathcal{H}^1_{\et}(G)$ is unramified if and only if the lower horizontal map is an isomorphism.
\end{proof}



\begin{remark}
  Purity for local rings $\mathcal{O}_{X,x}$ of smooth $F$-schemes would follow from the above if the Zariski sheafification of the presheaf $U\mapsto {\rm H}^1_{\et}(U,G)$ agrees with the Nisnevich sheafification $\mathcal{H}^1_{\et}(G)$. Conversely, if purity holds for local rings, then the Zariski sheafification agrees with $\mathcal{H}^1_{\et}(G)$ (and is unramified).
\end{remark}

Now at this point, we cannot prove that $\mathcal{H}^1_{\et}(G)$ is unramified in general. However, we would like to point out a few facts that strongly suggest Nisnevich local purity should be true: there is a morphism $\mathcal{H}^1_{\et}(G)\to \mathcal{T}$ to an unramified and homotopy invariant Nisnevich sheaf $\mathcal{T}$ which induces bijections over essentially smooth $F$-schemes of dimension $\leq 1$. In the presence of group structures on the sheaves, this would be enough to imply that the morphism is an isomorphism, but since we only have sheaves of pointed sets, we don't know how to prove the surjectivity.

For the construction of the unramified sheaf $\mathcal{T}$, it is clear that the sections over a smooth $F$-scheme have to be the unramified sections of $\mathcal{H}^1_{\et}(G)$:
\[
\mathcal{T}(X):=\bigcap_{x\in X^{(1)}} \mathcal{H}^1_{\et}(G)(\Spec \mathcal{O}_{X,x})\subseteq \mathcal{H}^1_{\et}(G)(\Spec K).
\]
For the presheaf structure, i.e., the existence of suitable restriction maps, we have to take a little detour. 

\begin{remark}
Applying the above definition to an arbitrary presheaf $\mathcal{F}$ via
\[
\mathcal{F}_{\rm nr}(X):=\bigcap_{x\in X^{(1)}} \mathcal{F}(\mathcal{O}_{X,x})\subseteq \mathcal{F}(K).
\]
doesn't generally produce a presheaf. The immediate problem is with the definition of restriction maps for non-smooth morphisms of smooth schemes. Conditions for the functoriality of the unramified sections were already discussed in \cite{colliot-thelene:sansuc}*{Section 6}. We will discuss another approach to the functoriality of unramified sections which makes use of purity in codimension 2 and the theory of unramified sheaves of \cite{MField}*{Section 2}. 
\end{remark}

\begin{proposition}
  \label{prop:unramified-datum}
  Let $F$ be a field and let $\mathcal{H}$ be a finitary Nisnevich sheaf on ${\rm Sm}_F$ having the strong Grothendieck--Serre property. Assume that $\mathcal{H}$ satisfies purity in dimension 2, i.e., for any irreducible essentially smooth $F$-scheme $X$ of dimension $2$ with function field $K$, we have an identification
  \[
  \mathcal{H}(X)=\bigcap_{x\in X^{(1)}}\mathcal{H}(\Spec\mathcal{O}_{X,x})
  \]
  of subsets in $\mathcal{H}(\Spec K)$.

  Then the following defines an unramified $\mathcal{F}_F$-datum in the sense of \cite{MField}*{Definition 2.9}.
  \begin{enumerate}[(D1)]
  \item We define a continuous contravariant functor $\mathcal{S}$ on the category $\mathcal{F}_F$ of separable field extensions of $F$ by $\mathcal{S}(L/F):=\mathcal{H}(L/F)$. 
  \item For any separable field extension $L/F$ and any discrete valuation $v$ on $L$ with valuation ring $\mathcal{O}_v$, we define $\mathcal{S}(\mathcal{O}_v):=\mathcal{H}(\Spec \mathcal{O}_v)\subseteq \mathcal{H}(\Spec L)=\mathcal{S}(L/F)$.
  \item For any separable field extension $L/F$ and any discrete valuation $v$ on $L$ with valuation ring $\mathcal{O}_v$ and residue field $\kappa(v)$, we define the specialization map $s\colon \mathcal{S}(\mathcal{O}_v)\to \mathcal{S}(\kappa(v))$ to be the restriction map $\mathcal{H}(\Spec \mathcal{O}_v)\to\mathcal{H}(\Spec \kappa(v))$. 
  \end{enumerate}
\end{proposition}

\begin{proof}
  We first prove that $\mathcal{S}$ is an unramified $\mathcal{F}_F$-datum, i.e., we need to check the axioms (A1)-(A4) from \cite{MField}*{Definition 2.6, 2.9}. We first note that the continuity of the functor $\mathcal{S}$ follows from the assumption that $\mathcal{H}$ is finitary. 
  
  Axiom (A1) is essentially a Nisnevich descent statement for extensions of dvrs, and it holds for $\mathcal{S}$ since $\mathcal{H}$ is a Nisnevich sheaf. In particular, the required map $\mathcal{S}(\mathcal{O}_w)\to\mathcal{S}(\mathcal{O}_v)$ is the restriction map for $\mathcal{H}$, and the sheaf property implies that the relevant square is cartesian.

  Axiom (A2) follows from the assumption that $\mathcal{H}$ is finitary. If $X$ is an irreducible smooth $F$-scheme with function field $K$, any section $\sigma\in \mathcal{S}(K)$ is already defined over some open $U\subseteq X$. Therefore, there are only finitely many $x\in X^{(1)}$ (the ones in the complement of $U$) such that $\sigma\not\in \mathcal{S}(\mathcal{O}_{X,x})$.

  Both parts of Axiom (A3) follow from the fact that $\mathcal{H}$ is a presheaf on the category of smooth schemes: $\mathcal{S}$ is defined in such a way that the commutativity of the diagram in part (i) resp. the claim about the induced maps in part (ii) follow from the identification of $\mathcal{S}$ in terms of $\mathcal{H}$ and then applying the presheaf $\mathcal{H}$ to an appropriate commutative diagram in the category of essentially smooth $F$-schemes.

  The non-trivial input is now in the proof of Axiom (A4). Let $X$ be any essentially smooth local $F$-scheme of dimension $2$ with closed point $z$ and function field $K$.  For part (i), assume that $y_0\in X^{(1)}$ is a codimension one point with essentially smooth closure. We need to show that the specialization map $\mathcal{S}(\mathcal{O}_{y_0})\to\mathcal{S}(\kappa(y_0))$ maps $\bigcap_{y\in X^{(1)}}\mathcal{S}(\mathcal{O}_y)$ into $\mathcal{S}(\mathcal{O}_{\overline{y_0},z})$. By construction, the specialization map is identified with the restriction map $\mathcal{H}(\mathcal{O}_{y_0})\to \mathcal{H}(\kappa(y_0))$, and 
  \[
  \bigcap_{y\in X^{(1)}}\mathcal{S}(\mathcal{O}_y)=\bigcap_{y\in X^{(1)}}\mathcal{H}(\mathcal{O}_y)\subset\mathcal{H}(K).
  \]
  By the assumption on $\mathcal{H}$, this set is identified with $\mathcal{H}(X)$. In particular, the specialization map will send the set $\mathcal{H}(X)\subseteq \mathcal{H}(\mathcal{O}_{y_0})$ into $\mathcal{H}(\mathcal{O}_{\overline{y_0},z})$ which proves part (i) of (A4). For part (ii), we note that the composition
  \[
  \bigcap_{y\in X^{(1)}}\mathcal{H}(\Spec \mathcal{O}_y)\to \mathcal{H}(\Spec \mathcal{O}_{\overline{y_0},z})\to\mathcal{H}(\Spec \kappa(z))
  \]
  is (by what we said for part (i)) simply the restriction $\mathcal{H}(X)\to \mathcal{H}(\Spec \kappa(z))$ and therefore independent of the choice of $y_0$.

\end{proof}

\begin{proposition}
  \label{prop:unramification}
In the situation of Proposition~\ref{prop:unramified-datum}, we have an unramified Nisnevich sheaf  $\mathcal{H}_{\rm nr}$   corresponding to $\mathcal{S}$. The restriction of sections from irreducible schemes to their function fields defines an injective morphism of sheaves $\mathcal{H}\to \mathcal{H}_{\rm nr}$ which induces bijections on essentially smooth $F$-schemes of dimension $\leq 1$.  The sheaf $\mathcal{H}_{\rm nr}$ has the strong Grothendieck--Serre property. If $\mathcal{H}$ is homotopy invariant then so is $\mathcal{H}_{\rm nr}$.
\end{proposition}

\begin{proof}
From Proposition~\ref{prop:unramified-datum} we see that $\mathcal{S}$ is an unramified $\mathcal{F}_F$-datum, and by \cite{MField}*{Theorem 2.11} there exists an unramified sheaf $\mathcal{H}_{\rm nr}$ associated to it. More precisely, the restriction of $\mathcal{H}_{\rm nr}$ (given exactly the same way we defined $\mathcal{S}$ from $\mathcal{H}$ in the first place) is isomorphic to $\mathcal{S}$. In particular, this means that for any irreducible essentially smooth $F$-scheme $X$ with function field $K$ we have 
  \[
  \mathcal{H}_{\rm nr}(X)=\bigcap_{x\in X^{(1)}}\mathcal{S}(\mathcal{O}_{X,x})\subseteq\mathcal{S}(K/F).
  \]
  
  We define the morphism $\mathcal{H}\to \mathcal{H}_{\rm nr}$ as follows: for an irreducible essentially smooth $F$-scheme with function field $K$, any section $\sigma\in\mathcal{H}(X)\subseteq\mathcal{H}(\Spec K)$ already lies in
  \[
  \bigcap_{x\in X^{(1)}}\mathcal{H}(\Spec \mathcal{O}_{X,x})\subseteq\mathcal{H}(\Spec K),
  \]
  and we send $\sigma\in \mathcal{H}(X)$ to its image under the identification
  \[
  \bigcap_{x\in X^{(1)}}\mathcal{H}(\Spec \mathcal{O}_{X,x})=\bigcap_{x\in X^{(1)}}\mathcal{S}(\mathcal{O}_{X,x})=\mathcal{H}_{\rm nr}(X)
  \]
  This defines maps $\mathcal{H}(X)\to \mathcal{H}_{\rm nr}(X)$ for any essentially smooth $F$-scheme $X$, and these maps are clearly injective since $\mathcal{H}(X)\subseteq\mathcal{H}(\Spec K)$. By construction, this produces bijections $\mathcal{H}(Y)\xrightarrow{\cong}\mathcal{H}_{\rm nr}(Y)$ for any irreducible essentially smooth $F$-scheme $Y$ of dimension $\leq 1$ (i.e., field extensions of $F$ and dvrs containing $F$). It therefore remains to show that these maps are compatible with the restriction maps for $\mathcal{H}$ resp. $\mathcal{H}_{\rm nr}$.

  The restriction maps of the sheaf $\mathcal{H}_{\rm nr}$ are constructed in the proofs of Proposition 2.8 (for smooth morphisms) and Theorem 2.11 (for closed immersions) of \cite{MField}. The restriction maps for smooth morphisms $X\to Y$ of essentially smooth $F$-schemes are induced by the function field extensions and therefore compatible with the restriction maps on $\mathcal{H}$ since $\mathcal{H}$ and $\mathcal{H}_{\rm nr}$ agree on schemes of dimension $\leq 1$. The definition of the restriction maps for closed immersions  reduces to the codimension 1  case by factoring a closed immersion into a sequence of codimension 1 immersions in \cite{MField}*{Lemma 2.13}. In particular, the uniqueness part of Lemma 2.13 implies that the restriction maps of $\mathcal{H}$ and $\mathcal{H}_{\rm nr}$ are compatible for closed immersion if they are compatible for codimension 1 closed immersions. But the restricition maps of $\mathcal{H}_{\rm nr}$ for codimension 1 closed immersions are defined in terms of the specialization maps for $\mathcal{S}$, hence are compatible with the restriction maps for $\mathcal{H}$. This concludes the proof.

  The strong Grothendieck--Serre property for $\mathcal{H}_{\rm nr}$ is then built in by construction. Unramified sheaves are also finitary (since the associated $\mathcal{F}_F$-datum satisfies the axiom (A2), cf. the discussion before \cite{MField}*{Proposition 2.8}. Assume that $\mathcal{H}$ satisfies homotopy invariance. Since $\mathcal{H}$ and $\mathcal{H}_{\rm nr}$ agree over one-dimensional schemes, $\mathcal{H}_{\rm nr}$ will then satisfy the $\mathbb{A}^1$-invariance over fields. Homotopy invariance for $\mathcal{H}_{\rm nr}$ follows from Proposition~\ref{prop:lnis-hoinv}. 
\end{proof}

\begin{remark}
  The result applies to $\mathcal{H}^1_{\et}(G)$ for a reductive group $G$ which is a finitary Nisnevich sheaf by construction and satisfies 2-dimensional purity by \cite{colliot-thelene:sansuc}*{Theorem 6.13, Corollary 6.14}. It seems very likely that the morphism $\mathcal{H}^1_{\et}(G)\to\mathcal{T}$ should be an isomorphism whenever the group $G$ satisfies $\mathbb{A}^1$-invariance over fields, but at this point we're unable to prove this.
\end{remark}

\begin{example}
  \label{ex:a1-rigid}
  There exist finitary homotopy-invariant Nisnevich sheaves of pointed sets which have the strong Grothendieck--Serre property but fail to be unramified (caused by a failure of unramifiedness in codimension 2). Consider for example the variety $X=(C_1\times C_2)\setminus \{p\}$ for $C_1$ and $C_2$ smooth projective curves of genus $>0$ and $p$ a closed point of the product $C_1\times C_2$. This variety is $\mathbb{A}^1$-rigid. The presheaf $\mathcal{X}$ represented by $X$ is then a finitary homotopy-invariant Nisnevich sheaf. The strong Grothendieck--Serre property is also satisfied. For a regular local ring $R$, the images of any two morphisms $\Spec R\to X$ lie in a common open affine  $\Spec A=U\subseteq X$, and two ring homomorphisms $A\to R$ agree if their compositions with $R\subseteq {\rm Frac}(R)$ agree. If the sheaf $\mathcal{X}$ were unramified, then for any regular local ring $R$ with fraction field $K$ a map $\Spec K \to X$ extends to $R$ if it extends to all codimension 1 points of $\Spec R$. However, by definition this fails for the regular local ring $\mathcal{O}_{C_1\times C_2,P}$ and the map $\Spec {\rm Frac}(\mathcal{O}_{C_1\times C_2,P})\to X$ induced from the obvious inclusion $\Spec \mathcal{O}_{C_1\times C_2,P}\hookrightarrow C_1\times C_2$. Similar arguments apply generally to open subsets of $\mathbb{A}^1$-rigid smooth projective varieties whose complement has codimension $\geq 2$.
\end{example}

\section{Connected components of classifying spaces}
\label{sec:components}

We now want to relate the sheaf $\mathcal{H}^1_{\et}(G)$ to the sheaf $a_{\rm Nis}\pi_0^{\mathbb{A}^1}{\rm B}_{\et}G$ of $\mathbb{A}^1$-connected components of the classifying space ${\rm B}_{\et}G$ for a reductive group $G$. More precisely, we will show that the natural morphism of sheaves in the following construction is an isomorphism.

\begin{construction}\label{cons:can} We construct a transformation of Nisnevich sheaves of pointed sets
\[
{\rm can}_G:\mathcal{H}^1_{\et}(-,G) \rightarrow a_{\rm Nis}\pi_0^{\mathbb{A}^1}({\rm B}_{\et}G).
\]
To begin with, consider the stack ${\rm B}_{\et}G$ classifying $G$-torsors. We have a morphism of presheaves ${\rm B}_{\et}G|_{\rm Sm_k} \rightarrow L_{\rm mot}({\rm B}_{\et}G|_{\rm Sm_k})$, which induces a map by taking $\pi_0$:
\[
{\rm H}^1_{\et}(-,G) \rightarrow \pi_0^{\mathbb{A}^1}({\rm B}_{\et}G).
\]
This is a transformation of presheaves of pointed sets, which then Nisnevich-sheafifies to the map ${\rm can}_G$ as above. 
\end{construction}

\begin{remark}\label{rem:torsor} Let $X \in {\rm Sm}_k$, the map ${\rm H}^1_{\et}(X,G) \rightarrow \pi_0^{\mathbb{A}^1}({\rm B}_{\et}G)(X)$ can be interpreted as follows: a $G$-torsor $\mathcal{P} \rightarrow X$ is classified by a map $X \rightarrow {\rm B}_{\et}G|_{\rm Sm_k}$ (here $X$ is abusively identified with its Yoneda image) which is sent to the motivic-localization 
\[
(L_{\rm mot}X \rightarrow L_{\rm mot}{\rm B}_{\et}G|_{\rm Sm_k}) \in \pi_0^{\mathbb{A}^1}({\rm B}_{\et}G)(X) = [X, {\rm B}_{\et}G]_{\mathbb{A}^1}.
\] 
In \cite{morel:voevodsky}*{Section 4.2}, the motivic localization of the presheaf ${\rm B}_{\et}G|_{\rm Sm_k}$ admits a geometric model (as an ind-scheme) which goes back to the work of Morel--Voevodsky and Totaro. Roughly there exists an ind-scheme $U_{\infty}$ with a $G$-action such that the quotient $(U_{\infty}/G)$ exists as an ind-scheme and we have an equivalence of presheaves:
\[
L_{\rm mot}B_{\et}G|_{\rm Sm_k} \simeq L_{\rm mot}(U_{\infty}/G)
\]
We refer to \cite{hoyois-cdh}*{Theorem 2.7} for a statement in this form, though we will not need it in the sequel.
\end{remark}

%

We now begin to examine the map ${\rm can}_G:\mathcal{H}^1_{\et}(-,G) \rightarrow a_{\rm Nis}\pi_0^{\mathbb{A}^1}({\rm B}_{\et}G)$. 


\begin{lemma}\label{lem:curves} Let $\mathcal{X}:{\rm Sm}^{\rm op}_k \rightarrow {\rm Spc}$ be an $\mathbb{A}^1$-invariant presheaf satisfying the following assumption:
\begin{itemize}
\item[(${\rm Surj}_{\leq1}$)] for any smooth, irreducible $k$-scheme $C$ of dimension $\leq 1$, the map
\[
\pi_0(\mathcal{X}(C)) \rightarrow \pi_0(L_{\rm Nis}\mathcal{X}(C)),
\]
is surjective.
\end{itemize}
Then for any finitely generated field extension $L$ of $k$, the map
\[
\pi_0(L_{\rm Nis}\mathcal{X}(L)) \rightarrow \pi_0({\rm Sing}L_{\rm Nis}\mathcal{X}(L))
\]
is injective.
\end{lemma}

\begin{proof} For any presheaf $\mathcal{Y}$, the set $\pi_0({\rm Sing}\mathcal{Y}(L))$ is explicitly presented as the set $\pi_0(\mathcal{Y}(L))$ modulo the equivalence relation generated by $\mathbb{A}^1$-homotopies. To prove injectivity, given two elements $x, y \in \pi_0(L_{\rm Nis}\mathcal{X}(L)) = \pi_0(\mathcal{X}(L))$ which are connected by a chain of $\mathbb{A}^1$-homotopies, we need to prove that $x = y$. Without loss of generality, we may assume that $x$ and $y$ are connected by a single $\mathbb{A}^1$-homotopy, i.e., there exists a map 
\[
h:\mathbb{A}^1_L \rightarrow L_{\rm Nis}\mathcal{X},
\]
such that precomposing with $\Spec L \xrightarrow{0} \mathbb{A}^1_L$ (resp. $\Spec L \xrightarrow{1} \mathbb{A}^1_L$) is the section $x$ (resp. $y$). But by the assumption (${\rm Surj}_{\leq1}$), there exists a map $\widetilde{h}: \mathbb{A}^1_L  \rightarrow \mathcal{X}$ which, when composed with the map $\mathcal{X} \rightarrow L_{\rm Nis}\mathcal{X}$ is $h$. Since $\mathcal{X}$ is $\mathbb{A}^1$-invariant, we conclude.
\end{proof}

\begin{lemma}\label{lem:bg}
    Let $F$ be a field and let $G$ be a reductive group satisfying $\mathbb{A}^1$-invariance for finitely generated extension fields $L/F$. The presheaf ${\rm Sing}{\rm B}_{\et}G$ satisfies $({\rm Surj}_{\leq1})$.
\end{lemma}

\begin{proof} An element $\alpha \in \pi_0(L_{\rm Nis}{\rm Sing}{\rm B}_{\et}G(C))$ is represented by a map $C \rightarrow L_{\rm Nis}{\rm Sing}{\rm B}_{\et}G$. Let $\eta: \Spec k(C) \rightarrow C$ be the generic point, then we have an element
  \[
  \alpha|_{\eta} \in \pi_0(L_{\rm Nis}{\rm Sing}{\rm B}_{\et}G(k(C))) = \pi_0({\rm Sing}{\rm B}_{\et}G(k(C))).
  \]
  Since $\pi_0(L_{\rm Nis}{\rm Sing}{\rm B}_{\et}G)$ is finitary, $\alpha|_{\eta}$ extends to a section $\alpha_U \in\pi_0(L_{\rm Nis}{\rm Sing}{\rm B}_{\et}G(U))$ and the reduced complement $C \setminus U$ consists of finitely many points. Suppose that the complement consists of only a singleton $\{x\}$. Consider the Nisnevich square
  \[
  \xymatrix{
W= \Spec {\rm Frac(\mathcal{O}^h_{X,x})}   \ar[r] \ar[d] & \Spec \mathcal{O}^h_{X,x}\ar[d] \\
    U \ar[r] & C.
  }
  \]
  Then the sections 
  \[
  \alpha|_{ \mathcal{O}^{\rm h}_{X,x}} \in \pi_0(L_{\rm Nis}{\rm Sing}{\rm B}_{\et}G(\mathcal{O}^{\rm h}_{X,x})) = \pi_0({\rm Sing}{\rm B}_{\et}G(\mathcal{O}^{\rm h}_{X,x})) \qquad \alpha_U \in\pi_0({\rm Sing}{\rm B}_{\et}G(U)),
  \]
 both agree on $\pi_0(L_{\rm Nis}{\rm Sing}{\rm B}_{\et}G(W)) = \pi_0({\rm Sing}{\rm B}_{\et}G(W))$. This means that the elements $\alpha|_{ \mathcal{O}^{\rm h}_{X,x}}$ and $\alpha_U$ are represented as $G$-torsors over their respective base which agrees up to $\mathbb{A}^1$-homotopy on $W$. But now $W$ is the spectrum of a field and hence we can use assumed $\mathbb{A}^1$-invariance over fields to conclude that these torsors are actually isomorphic on $W$. Since ${\rm B}_{\et}G$ is a Nisnevich sheaf, we conclude.
\end{proof}

\begin{proposition}
  \label{prop:inj-f}
  Let $F$ be a field and let $G$ be a reductive group satisfying $\mathbb{A}^1$-invariance for finitely generated extension fields $L/F$. Then the morphism ${\rm can}_G$ induces injections
  \[
    {\rm can}_G\colon\mathcal{H}^1_{\et}(G)(\Spec L)\hookrightarrow \left(a_{\rm Nis}\pi_0^{\mathbb{A}^1}{\rm B}_{\et}G\right)(\Spec L)
  \]
  for all finitely generated extension fields $L/F$.
\end{proposition}

\begin{proof}
  Let $L$ be a finitely generated extension of $F$. Since $a_{\rm Nis}$ does not change values of sheaves on fields, we need to prove that the map
  \[
    {\rm H}^1_{\et}(\Spec L,G)\hookrightarrow \pi_0^{\mathbb{A}^1}{\rm B}_{\et}G(\Spec L)=\pi_0(L_{\rm mot}{\rm B}_{\et}G(\Spec L))
  \]
  is injective. From the description of $L_{\rm mot}$ as a filtered colimit of the functors $L_{\rm Nis}$ and $L_{\mathbb{A}^1}=\op{Sing}_\bullet^{\mathbb{A}^1}$ (see, for example, \cite{antieau:elmanto}), it suffices to prove that the morphism
  \[
    {\rm H}^1_{\et}(\Spec L,G)\hookrightarrow \pi_0\left(\left(L_{\rm Nis}\circ \op{Sing}_{\bullet}^{\mathbb{A}^1}\right)^n{\rm B}_{\et}G\right)(\Spec L)
  \]
  is injective for all natural numbers $n$. The claim of the proposition then follows from the fact that filtered colimits preserve monomorphisms.

  We will show by induction on $n$ that
  \begin{enumerate}[(i)]
  \item the morphism ${\rm H}^1_{\et}(C,G)\to\pi_0\left(\left(\op{Sing}_{\bullet}^{\mathbb{A}^1}\circ L_{\rm Nis}\right)^n{\rm B}_{\et}G\right)(C)$  is surjective for all essentially smooth $F$-schemes of dimension $\leq 1$,
  \item the morphism ${\rm H}^1_{\et}(L,G)\to \pi_0\left(\left(\op{Sing}_{\bullet}^{\mathbb{A}^1}\circ L_{\rm Nis}\right)^n{\rm B}_{\et}G\right)(L)$ is injective for all finitely generated field extensions of $F$.
  \end{enumerate}

  The base of the induction is the case $n=1$. In this case, (i) is proved in Lemma~\ref{lem:bg}. For (ii), we note that the target of the map
\[
{\rm H}^1_{\et}(L; G) \rightarrow \pi_0(\op{Sing}^{\mathbb{A}^1}_\bullet{\rm B}_{\et}G(L))
\]
is the coequalizer of the maps
\[
i_0^*, i_1^*: {\rm H}^1_{\et}(\mathbb{A}^1_L; G) \rightarrow {\rm H}^1_{\et}(L; G). 
\]
Therefore the claim (ii) is equivalent to the following statement:
\begin{itemize}
\item if $\mathcal{P}, \mathcal{P'}$ are two $G$-torsors over $\Spec L$ which are $\mathbb{A}^1$-homotopic, then they are, in fact, isomorphic as $G$-torsors.
\end{itemize}
This follows from the assumption of $\mathbb{A}^1$-invariance of ${\rm H}^1_{\et}(-,G)$ over $L$: any $G$-torsor over $\mathbb{A}^1_L$ is extended from $\Spec L$ and therefore the fibers over $0$ and $1$ are isomorphic. 

Now for the induction step, assume (i) and (ii) hold for the $n$-th iteration of $\op{Sing}_\bullet^{\mathbb{A}^1}\circ L_{\rm Nis}$; we want to show they continue to hold for the $(n+1)$-th iteration. For this,  we will prove that ${\rm H}^1_{\et}(C,G)\to\pi_0\left(L_{\rm Nis}\left(\op{Sing}_{\bullet}^{\mathbb{A}^1}\circ L_{\rm Nis}\right)^n{\rm B}_{\et}G\right)(C)$  is surjective for essentially smooth $F$-schemes of dimension $\leq 1$. Then (ii) will follow using Lemma~\ref{lem:curves} and the inductive hypothesis (ii); and (i) will follow since applying $\op{Sing}_\bullet^{\mathbb{A}^1}$ only introduces $\mathbb{A}^1$-relations and is hence surjective. The required surjectivity of ${\rm H}^1_{\et}(C,G)\to\pi_0\left(L_{\rm Nis}\left(\op{Sing}_{\bullet}^{\mathbb{A}^1}\circ L_{\rm Nis}\right)^n{\rm B}_{\et}G\right)(C)$ now follows by an argument exactly as in Lemma~\ref{lem:bg}, using the inductive hypothesis (i) to represent sections of $\pi_0\left(\left(\op{Sing}_{\bullet}^{\mathbb{A}^1}\circ L_{\rm Nis}\right)^n{\rm B}_{\et}G\right)$ by $G$-torsors.

By induction, (ii) holds for all natural $n$, and this establishes the claimed injectivity.
\end{proof}


%
%

\begin{proposition}\label{prop:inj-sm}
  Let $F$ be a field and let $G$ be a reductive group satisfying $\mathbb{A}^1$-invariance for finitely generated extension fields $L/F$. The  morphism ${\rm can}_G$ is injective over any smooth $k$-scheme.
\end{proposition}

\begin{proof}
  Let $X \in {\rm Sm}_k$. Since we have a bijection over the function field $K$ of $X$ from Proposition~\ref{prop:inj-f}, the result follows from the commutative square
\begin{equation}\label{eq:gssquare}
\xymatrix{
  \mathcal{H}^1_{\et}(G)(X)\ar[r] \ar@{^(->}[d] &  a_{\rm Nis}\pi_0^{\mathbb{A}^1}({\rm B}_{\et}G)(X) \ar[d] \\
  \mathcal{H}^1_{\et}(G)(\Spec K)\ar[r]_{\cong} &  a_{\rm Nis}\pi_0^{\mathbb{A}^1}({\rm B}_{\et}G)(\Spec K).
  }
\end{equation}
Note that we have used Proposition~\ref{prop:lnis-unramified-1} and the validity of the Grothendieck--Serre conjecture to get injectivity for the left vertical map.
\end{proof}


\begin{proposition}\label{prop:locsurj}
 The morphism ${\rm can}_G$ is surjective on henselian regular local rings
\end{proposition}

\begin{proof}
  The morphism is surjective on local rings for the Nisnevich topology because $\pi_0({\rm B}_{\et} G)\to \pi_0^{\mathbb{A}^1}({\rm B}_{\et}G)$ (map from simplicial $\pi_0$) is surjective after Nisnevich sheafification, cf.~\cite{morel:voevodsky}*{Corollary 2.3.22}. Therefore it is surjective on henselian regular local rings $R$ and any map $\Spec R\to {\rm B}_{\et}G$ is represented by an \'etale $G$-torsor over $\Spec R$.
\end{proof}

\begin{theorem}
  \label{thm:h1et-vs-pi0a1}
  Let $F$ be a field and let $G$ be a reductive group satisfying $\mathbb{A}^1$-invariance for finitely generated extension fields $L/F$. Then the natural map $\mathcal{H}^1_{\et}(G)\to a_{\rm Nis}\pi_0^{\mathbb{A}^1}({\rm B}_{\et}G)$ is an isomorphism of sheaves. In particular, $a_{\rm Nis}\pi_0^{\mathbb{A}^1}({\rm B}_{\et}G)$ is a homotopy invariant Nisnevich sheaf.
\end{theorem}

\begin{proof}
  The first statement follows from Propositions~\ref{prop:locsurj} and~\ref{prop:inj-sm}.
\end{proof}

\begin{remark}
  The above now establishes Morel's conjecture for classifying spaces of reductive groups $G$ having $\mathbb{A}^1$-invariance over fields. If in addition local purity holds for $G$, the sheaf $a_{\rm Nis}\pi_0^{\mathbb{A}^1}{\rm B}_{\et}G$ is also unramified, as a consequence of Theorem~\ref{thm:h1et-vs-pi0a1} and Lemma~\ref{lem:purity-unramified}. However, it may be worth pointing out that while homotopy invariance is obviously expected for the sheaves $a_{\rm Nis}\pi_0^{\mathbb{A}^1}(\mathcal{X})$ of $\mathbb{A}^1$-connected components for arbitrary motivic spaces $\mathcal{X}$, these sheaves are not expected to be unramified in general. Example~\ref{ex:a1-rigid} shows that the sheaves of $\mathbb{A}^1$-connected components of non-proper $\mathbb{A}^1$-rigid varieties usually fail to be unramified. In some sense, being unramified is a weak properness statement.
\end{remark}

\section{Applications}
\label{sec:examples}

In this final section, we will discuss some consequences of our results on $\mathcal{H}^1_{\et}(G)$ and $a_{\rm Nis}\pi_0^{\mathbb{A}^1}{\rm B}_{\et}G$ and how these relate to torsor classification and purity questions. For the rest of the section we will always only talk about groups $G$ where the conditions for Theorem~\ref{thm:h1et-vs-pi0a1} are satisfied and $\mathcal{H}^1_{\et}(G)$ is unramified. 

\subsection{Computation of the sheaves of $\mathbb{A}^1$-connected components}
If $\mathcal{H}^1_{\et}(G)$ is unramified, then the identification $\mathcal{H}^1_{\et}(G)\cong a_{\rm Nis}\pi_0^{\mathbb{A}^1}{\rm B}_{\et}G$ allows to us to compute the sheaf of $\mathbb{A}^1$-connected components of the classifying space ${\rm B}_{\et}G$. For an irreducible smooth scheme $X$ with function field $K$, the sections of $a_{\rm Nis}\pi_0^{\mathbb{A}^1}{\rm B}_{\et}G$ over $X$ are given by $G$-torsors over $\Spec K$ which extend over all the codimension 1 points of $X$. This classification question is significantly easier than the isomorphism classification of $G$-torsors over $X$ itself and can be made rather explicit in specific cases.  It turns out that these sheaves are \emph{non-trivial} in many cases, in contrast to the case of special groups. 

Our first examples are orthogonal groups. 


\begin{proposition}
  \label{prop:components-on}
  Let $F$ be a field of characteristic $\neq 2$. Then for any irreducible smooth $F$-scheme with function field $K$, we have
  \[
  \left(a_{\rm Nis}\pi_0^{\mathbb{A}^1}{\rm B}_{\et}{\rm O}(n)\right)(X)={\rm GW}_{\rm nr}(X)\cap {\rm H}^1_{\et}(K,{\rm O}(n))
  \]
  where the intersection is taken inside ${\rm GW}(K)$. This means that the sheaf $a_{\rm Nis}\pi_0^{\mathbb{A}^1}{\rm B}_{\et}{\rm O}(n)$ of $\mathbb{A}^1$-connected components of the classifying space of ${\rm O}(n)$ is the subsheaf of the sheaf of unramified Grothendieck--Witt groups consisting of classes of quadratic forms over $\Spec K$ which have a rank $n$ representative.
\end{proposition}

\begin{proof}
  By Theorem~\ref{thm:h1et-vs-pi0a1} it suffices to compute
  \[
  \mathcal{H}^1_{\et}({\rm O}(n))(X)=\bigcap_{x\in X^{(1)}}\mathcal{H}^1_{\et}({\rm O}(n))(\mathcal{O}_{X,x})=\bigcap_{x\in X^{(1)}}{\rm H}^1_{\et}(\mathcal{O}_{X,x},{\rm O}(n))\subseteq{\rm H}^1_{\et}(K,{\rm O}(n)),
  \]
  with the first equality being the unramifiedness of $\mathcal{H}^1_{\et}({\rm O}(n))$ from Lemma~\ref{lem:purity-unramified} and the purity result of Panin \cite{panin:purity} resp. its extension in \cite{panin:pimenov}, and the second equality follows as in the proof of Lemma~\ref{lem:purity-unramified}. We note that the natural map
  \[
    {\rm H}^1_{\et}(K,{\rm O}(n))\to {\rm GW}(K)
    \]
    taking an isometry class of a rank $n$ quadratic form over $K$ to its class in the Grothendieck--Witt ring is injective by Witt cancellation. Now a rank $n$ quadratic form extends over a dvr $\mathcal{O}_{X,x}\subseteq K$ if its class in the Witt ring of $K$ is unramified by \cite{scharlau}*{Theorem 2.2, Chapter 6}.\footnote{Note that the boundary of a class in ${\rm GW}(K)$ being zero means in the notation of loc.cit. that we can choose a lattice $L$ with $L^\#/L=0$. Then $L$ with its induced form is the required form over the dvr.} This implies that $\mathcal{H}^1_{\et}({\rm O}(n))(X)$ is the intersection of ${\rm H}^1_{\et}(K,{\rm O}(n))$ and the unramified Grothendieck--Witt group ${\rm GW}_{\rm nr}(X)$ inside ${\rm GW}(K)$. This proves the claim.
\end{proof}

\begin{example}\label{ex:pgl}
  As another example for a description of the sheaf of $\mathbb{A}^1$-connected components of a classifying space, a similar result is true for ${\rm PGL}_n$ over a field $F$ of characteristic $0$. For any irreducible smooth $F$-scheme with function field $K$, we have
  \[
  \left(a_{\rm Nis}\pi_0^{\mathbb{A}^1}{\rm B}_{\et}{\rm PGL}_n\right)(X)={\rm Br}(X)\cap {\rm H}^1_{\et}(K,{\rm PGL}_n)
  \]
  with the intersection taken inside ${\rm Br}(K)={\rm H}^1_{\et}(K,{\rm PGL}_\infty)$. This uses that the Brauer group is unramified and that a ${\rm PGL}_n$-torsor over $K$ extends over a dvr $\mathcal{O}_{X,x}\subseteq K$ if and only if the local invariant is trivial. 
\end{example}

\begin{example}
  We can also describe the sheaf of $\mathbb{A}^1$-connected components for the classifying space of the exceptional group ${\rm G}_2$ over a field of characteristic $\neq 2$, using the local purity result of Chernousov and Panin \cite{chernousov:panin} resp. its extension in \cite{panin:pimenov}. Over a field $K$ of characteristic $\neq 2$, we have
  \[
  {\rm H}^1_{\et}(K,{\rm G}_2)\cong {\rm H}^3_{\et}(K,\mu_2)_{\rm dec},
  \]
  i.e., ${\rm G}_2$-torsors over $K$ are classified by 3-fold Pfister forms over $K$, cf. \cite{serre}*{Th{\'e}or{\`e}me 9}. Actually, there is a bijection between ${\rm G}_2$-torsors and 3-fold Pfister forms over local rings in which 2 is invertible, cf. \cite{chernousov:panin}*{Remark 10}. In particular, a ${\rm G}_2$-torsor over a discretely valued field $K$ extends over the valuation ring $R\subseteq K$ if and only if the norm form extends to a 3-fold Pfister form over $R$. Consequently, the sections of $a_{\rm Nis}\pi_0^{\mathbb{A}^1}{\rm B}_{\et}{\rm G}_2$ over an irreducible smooth scheme $X$ with function field $K$ can be identified with the isomorphism classes of unramified 3-fold Pfister forms over $K$.
\end{example}

\begin{remark}
  The above examples should convince the reader that the identification of $a_{\rm Nis}\pi_0^{\mathbb{A}^1}{\rm B}_{\et}G$ allows to determine the sheaves of $\mathbb{A}^1$-connected components of classifying spaces ${\rm B}_{\et}G$ fairly explicitly in a number of interesting cases. All that is required is knowledge about the torsor classification over fields and dvrs.
\end{remark}

\subsection{Purity for torsors} Next, we want to discuss the relation between our results and purity questions for torsors. If Nisnevich-local purity holds for $G$-torsors, i.e., $\mathcal{H}^1_{\et}(G)$ is unramified, we can identify it as the target of the map in the definition of purity, cf. Definition~\ref{def:purity}. This way, we can reformulate purity as the surjectivity of the sheafification map for \'etale torsors:

\begin{proposition}
  \label{prop:purity-sheafification}
  Let $F$ be a field and $G$ be a reductive group over $F$ such that $\mathcal{H}^1_{\et}(G)$ is unramified. Then purity for $G$-torsors over a smooth $F$-scheme $X$ is equivalent to the surjectivity of the sheafification map ${\rm H}^1_{\et}(X,G)\to \mathcal{H}^1_{\et}(G)(X)$. If purity holds for $G$-torsors over $X$, then the sheafification map $\pi_0^{\mathbb{A}^1}{\rm B}_{\et}G(X)\to a_{\rm Nis}\pi_0^{\mathbb{A}^1}{\rm B}_{\et}G(X)$ is surjective.
\end{proposition}

\begin{proof}
  The target of the purity map in Definition~\ref{def:purity} is identified with
\[
\mathcal{H}^1_{\et}(G)(X)=\bigcap_{x\in X^{(1)}}\mathcal{H}^1_{\et}(G)(\mathcal{O}_{X,x}).
\]
We also have an identification
\[
\mathcal{H}^1_{\et}(G)(\mathcal{O}_{X,x})\cong \op{im}\left({\rm H}^1_{\et}(\mathcal{O}_{X,x},G)\to {\rm H}^1_{\et}(K,G)\right)
\]
as in the proof of Lemma~\ref{lem:purity-unramified}. With this identification, purity for $G$-torsors on $X$ is equivalent to the surjectivity of the sheafification map
\[
{\rm H}^1_{\et}(X,G)\to \mathcal{H}^1_{\et}(G)(X).
\]
Now consider the following commutative diagram
\[
\xymatrix{
  {\rm H}^1_{\et}(X,G) \ar[r] \ar[d] & \pi_0^{\mathbb{A}^1}{\rm B}_{\et}G(X)\ar[d] \\
  \mathcal{H}^1_{\et}(G)(X) \ar[r]_{\cong} & a_{\rm Nis}\pi_0^{\mathbb{A}^1}{\rm B}_{\et}G (X)
}
\]
The lower horizontal map is an isomorphism by Theorem~\ref{thm:h1et-vs-pi0a1}. As noted above, purity for torsors is equivalent to the surjectivity of the left vertical morphism.  Consequently, if purity is satisfied for $G$-torsors over a scheme $X$, the sheafification map $\pi_0^{\mathbb{A}^1}{\rm B}_{\et}G(X) \to a_{\rm Nis}\pi_0^{\mathbb{A}^1}{\rm B}_{\et}G (X)$ is surjective.
\end{proof}

\begin{example}
  Examples of the failure of global purity for torsors under ${\rm PGL}_p$ have been given in \cite{antieau:williams}. By the above discussion, these provide, for any prime $p$, examples of smooth affine complex varieties $X$ of dimension $2p+2$ such that the sheafification map
  \[
  {\rm H}^1_{\et}(X,{\rm PGL}_p)\to \mathcal{H}^1_{\et}({\rm PGL}_p)(X)\cong \left(a_{\rm Nis}\pi_0^{\mathbb{A}^1}{\rm B}_{\et}{\rm PGL}_p\right)(X)
  \]
  fails to be surjective. In particular, we also have examples of sections of the sheaf of $\mathbb{A}^1$-connected components of ${\rm B}_{\et}{\rm PGL}_p$ which fail to be represented by actual torsors. In fact, they also fail to be represented by motivic torsors as the next theorem explains.
\end{example}

\begin{theorem}\label{thm:aw-example} Let $k$ be a field of characteristic zero. Then for any prime $p$, there exists a smooth affine $k$-scheme of dimension $2p+2$ such that the map 
 \[
\pi_0^{\mathbb{A}^1}{\rm B}_{\et}{\rm PGL}_p(X) \to \mathcal{H}^1_{\et}({\rm PGL}_p)(X)\cong \left(a_{\rm Nis}\pi_0^{\mathbb{A}^1}{\rm B}_{\et}{\rm PGL}_p\right)(X)
  \]
  fail to be surjective
\end{theorem}

\begin{proof} Let $p > 0$ be fixed. Let $X$ be as in \cite{antieau:williams}*{Theorem 3.6}; we note that while $X$ is over the complex numbers here, the construction works over any characteristic zero field. This is a smooth affine scheme of dimension $2p+2$ equipped with a Brauer class $\alpha \in {\rm Br}(X)$ such that $\alpha|_{k(X)}$ is exact degree $p$; this is classified by a map of presheaves 
\begin{equation}\label{eq:pre}
X \xrightarrow{\alpha} {\rm B}^2_{\et}\mathbb{G}_m.
\end{equation}
By the identification in Example~\ref{ex:pgl}, $\alpha$ determines an element $\alpha_{{\rm mot}}|_K \in \left(a_{\rm Nis}\pi_0^{\mathbb{A}^1}{\rm B}_{\et}{\rm PGL}_p\right)(X)$. If this element does lift to an element $\alpha_{\rm mot} \in \pi_0^{\mathbb{A}^1}{\rm B}_{\et}{\rm PGL}_p(X)$, then we have constructed a nontrivial factorization of~\eqref{eq:pre} in the $\mathbb{A}^1$-homotopy category:
\[
L_{\rm mot}X \rightarrow L_{\rm mot}{\rm B}_{\et}{\rm PGL}_p \rightarrow {\rm B}^2_{\et}\mathbb{G}_m.
\]
This then Betti realizes to a factorization of the Betti realization of~\eqref{eq:pre} which cannot exist as explained in \cite{antieau:williams}*{Theorem 3.6}, see also \cite{antieau:williams}*{Theorem 3.11}.
\end{proof}

\begin{remark}
  The failure of surjectivity for the sheafification map on $\pi_0^{\mathbb{A}^1}{\rm B}_{\et}{\rm PGL}_p$ is based on the fact that in the topological realization the difference between actual torsors and motivic torsors vanishes: the realization of ${\rm B}_{\et}{\rm PGL}_p$ is a classifying space for ${\rm PGL}_p(\mathbb{C})$ and thus counterexamples to topological purity also imply counterexamples to purity for motivic $G$-torsors. It is likely that similar constructions can be made in characteristic $\neq p$ using \'etale realization.
\end{remark}

\begin{remark}\label{rem:ind-schemes} Let us now consider $G = {\rm PGL_n}$ where $n$ is not necessarily a prime. Over a field of characteristic zero, we can find examples of motivic spaces $X$ where the sheafification map is not surjective $\pi_0^{\mathbb{A}^1}{\rm B}_{\et}G(X)\to a_{\rm Nis}\pi_0^{\mathbb{A}^1}{\rm B}_{\et}G(X)$. Unfortunately, these examples are motivic spaces associated to ind-schemes (they are approximations of ${\rm B}_{\et}G$ in the sense of Totaro and Morel-Voeovodsky \cite{morel:voevodsky}) and thus does not lead to new counterexamples to purity in this situation. These examples also come from the work Antieau-Williams \cite{antieau:williams}
\end{remark}

\begin{proposition}\label{prop:ind-schemes} Let $k$ be a field of characteristic zero, then there exists an ind-scheme $X$ such that
\[
\pi_0^{\mathbb{A}^1}{\rm B}_{\et}G(X)\to a_{\rm Nis}\pi_0^{\mathbb{A}^1}{\rm B}_{\et}G(X)
\]
is not surjective.
\end{proposition}

\begin{proof} Let $m > 1$ be an integer dividing $n$ and let $H$ be the algebraic group ${\rm SL_{m}}/\mu_n$. Let $X$ be an algebraic approximation to ${\rm B}_{\et}H$ as in \cite{morel:voevodsky}*{Proposition 4.2.6}. In particular $X$ comes equipped with a canonical ``Brauer class" $\alpha$; more precisely it is a map of presheaves $\alpha: X \rightarrow {\rm B}^2_{\et}\mu_n$ induced by the exact sequence of \'etale sheaf of groups $1 \rightarrow \mu_n \rightarrow {\rm SL}_m \rightarrow H \rightarrow 1$. By the identification in Example~\ref{ex:pgl}, extended by filtered colimits to ind-schemes, we $\alpha$ determines an element of $a_{\rm Nis}\pi_0^{\mathbb{A}^1}{\rm B}_{\et}G(X)$. If $\alpha$ does lift to $\pi_0^{\mathbb{A}^1}{\rm B}_{\et}G(X)$, then we would obtain a factorization in the $\mathbb{A}^1$-homotopy category of the map $\alpha$ as
\[
X \rightarrow {\rm B}_{\et}{\rm PGL_n} \rightarrow {\rm B}^2_{\et}\mu_n.
\]
But this contradicts \cite[Theorem 3.11]{antieau:williams} which states that no factorization can exist after taking Betti realizations. 
\end{proof}

\begin{example}
  Similarly, we can ask for the failure of global purity for ${\rm O}(n)$-torsors. From the description of $\mathcal{H}^1_{\et}({\rm O}(n))$ in Proposition~\ref{prop:components-on}, one source for the failure of purity for ${\rm O}(n)$ torsors is the failure of purity for the Grothendieck--Witt group. For an irreducible smooth scheme $X$ with function field $K$, any class in ${\rm W}_{\rm nr}(X)\cap {\rm H}^1_{\et}(K,{\rm O}(n))$ which is not in the image of the natural map ${\rm W}(X)\to {\rm W}_{\rm nr}(X)$ provides a counterexample to purity for ${\rm O}(n)$-torsors. The failure of surjectivity of the morphism ${\rm W}(X)\to {\rm W}_{\rm nr}(X)$ can be studied using the Gersten--Witt spectral sequence \cite{balmer:walter}. By the weak purity theorem of loc.cit., such phenomena cannot occur in dimensions $\leq 4$, and the unique obstruction to purity of Witt groups for smooth schemes of dimension $\leq 8$ is a differential ${\rm W}_{\rm nr}(X)\to {\rm H}^5(X,\mathbf{W})$ in the Gersten--Witt spectral sequence. Not much seems to be known about examples where this differential is non-trivial.
\end{example}

More generally, the formulation in Proposition~\ref{prop:purity-sheafification} now opens up the possibility to use motivic homotopy methods for investigation of counterexamples to the global purity question for $G$-torsors. By a conjecture of Antieau and Williams, cf. \cite{antieau:williams}*{Conjecture 1.2}, purity should fail for $G$-torsors over some smooth affine variety if $G$ is a non-special semisimple group. By Proposition~\ref{prop:purity-sheafification}, to get counterexamples to purity for such a group $G$, it would suffice to find examples of smooth affine schemes $X$ over a field $F$ such that the sheafification map
\[
\pi_0^{\mathbb{A}^1}{\rm B}_{\et}G(X) \to a_{\rm Nis}\pi_0^{\mathbb{A}^1}{\rm B}_{\et}G (X)
\]
fails to be surjective. We indicate how this could be done for ${\rm PGL_n}$ in the next remark, noting that it is subject to a future investigation.


\subsection{Towards the classification of (motivic) $G$-torsors} Finally, we can once more have a look at the diagram employed in the proof of Proposition~\ref{prop:purity-sheafification} to discuss the relations between torsor classification and the homotopy theory of the classifying spaces:
\[
\xymatrix{
  {\rm H}^1_{\et}(X,G) \ar[r] \ar[d] & \pi_0^{\mathbb{A}^1}{\rm B}_{\et}G(X)\ar[d] \\
  \mathcal{H}^1_{\et}(G)(X) \ar[r]_{\cong} & a_{\rm Nis}\pi_0^{\mathbb{A}^1}{\rm B}_{\et}G (X)
}
\]
The upper left corner is about the isomorphism classification of $G$-torsors on a smooth scheme $X$, the upper right corner is about the homotopy classification of maps into the geometric classifying spaces ${\rm B}_{\et}G$. The lower part of the diagram is about the sheafified problems which we now have identified as the classification of unramified torsors over the function field.

The left-hand vertical map relates the isomorphism classification of torsors over $X$ with the classification of unramified torsors over the function field. As noted in Proposition~\ref{prop:purity-sheafification}, surjectivity of that map is the question of purity for $G$-torsors over $X$, and that is expected to fail in general (though at this point we only have examples for ${\rm PGL}_p$). The sheafification map ${\rm H}^1_{\et}(X,G)\to \mathcal{H}^1_{\et}(G)(X)$ is also not going to be injective in general: all the rationally trivial torsors over $X$ map to the trivial unramified torsor in $\mathcal{H}^1_{\et}(G)(X)$. It would certainly be interesting to classify $G$-torsors over $X$ mapping to a given class in $\mathcal{H}^1_{\et}(G)(X)$, but at this point there don't seem to be methods around to approach this question.

The right-hand vertical map relates the presheaf of $\mathbb{A}^1$-connected components of ${\rm B}_{\et}G$ with its sheafification. By Theorem~\ref{thm:aw-example}, the sheafification map
\[
\pi_0^{\mathbb{A}^1}{\rm B}_{\et}G(X)\to a_{\rm Nis}\pi_0^{\mathbb{A}^1}{\rm B}_{\et}G(X)
\]
also fails to be surjective in general. The sheafification map is generally not injective, since the representability results from \cite{affine-representability-3} imply that the preimage of the basepoint in $a_{\rm Nis}\pi_0^{\mathbb{A}^1}{\rm B}_{\et}G(X)$ is given by isomorphism classes of rationally trivial torsors. Again, it would be interesting to understand the fibers of the sheafification map as well as conditions for realizability of classes in $a_{\rm Nis}\pi_0^{\mathbb{A}^1}{\rm B}_{\et}G(X)$ in terms of maps $X\to {\rm B}_{\et}G$. Possibly some version of obstruction theory for non-connected spaces could help here.

At last, the top horizontal map in the diagram relates the isomorphism classification of $G$-torsors to the $\mathbb{A}^1$-homotopy classification of maps $X\to {\rm B}_{\et}G$. This is essentially the question what the classifying space ${\rm B}_{\et}G$ actually classifies. Unfortunately, this map will not generally be a bijection. It fails to be injective because $\pi_0^{\mathbb{A}^1}{\rm B}_{\et}G$ is homotopy invariant whereas for non-special groups ${\rm H}^1_{\et}(-,G)$ is not generally homotopy invariant and therefore the map ${\rm H}^1_{\et}(X,G)\to \pi_0^{\mathbb{A}^1}{\rm B}_{\et}G(X)$ forgets about the counterexamples to homotopy invariance for \'etale torsors. At this point, it seems nothing is known regarding the surjectivity of the natural map ${\rm H}^1_{\et}(X,G)\to \pi_0^{\mathbb{A}^1}{\rm B}_{\et}G(X)$. For instance, we don't know if there exists a smooth affine scheme $X$ and a motivic ${\rm PGL}_n$-torsor over $X$ which isn't represented by an actual ${\rm PGL}_n$-torsor over $X$.

\end{document}